\documentclass[11pt,reqno]{amsart}
\usepackage[T1]{fontenc}
\usepackage{amsmath,amssymb,amsthm,mathtools}
\usepackage[a4paper,margin=29mm]{geometry}
\usepackage[colorlinks=true,linkcolor=blue,citecolor=blue,urlcolor=blue]{hyperref}
\hypersetup{pdftitle={A Near-Quadratic Lower Bound for Sets with No Unique Sums},pdfauthor={Jianfeng Hou and Kai Yang}}
\allowdisplaybreaks
\newtheorem{theorem}{Theorem}[section]
\newtheorem{lemma}[theorem]{Lemma}
\newtheorem{proposition}[theorem]{Proposition}
\newtheorem{corollary}[theorem]{Corollary}
\theoremstyle{definition}
\newtheorem{definition}[theorem]{Definition}
\newtheorem{question}[theorem]{Question}
\theoremstyle{remark}

\numberwithin{equation}{section}
\newcommand{\Z}{\mathbb Z}
\newcommand{\Q}{\mathbb Q}
\newcommand{\F}{\mathbb F}
\DeclareMathOperator{\rk}{rank}
\DeclareMathOperator{\supp}{supp}

\newcommand{\one}{\mathbf 1}
\newcommand{\ee}{\mathrm e}
\newcommand{\norm}[1]{\lVert#1\rVert}
\title[Near-quadratic bounds for sets with no unique sums]{A Near-Quadratic Lower Bound for Sets with No Unique Sums}
\author[J. Hou]{Jianfeng Hou\textsuperscript{1}}
\thanks{\textsuperscript{1}Fuzhou University. Email: \href{mailto:jfhou@fzu.edu.cn}{\nolinkurl{jfhou@fzu.edu.cn}}.}
\author[K. Yang]{Kai Yang\textsuperscript{2}}
\thanks{\textsuperscript{2}Fuzhou University. Email: \href{mailto:2600310016@fzu.edu.cn}{\nolinkurl{2600310016@fzu.edu.cn}}.}
\date{}
\subjclass[2020]{11B13, 11B34}
\keywords{Unique sums, finite fields, collision lattices, integer compression, balanced sets}
\begin{document}
\begin{abstract}
Let $m(p)$ be the least size of a subset of $\F_p$ with at least two
elements for which every sum has two distinct representations as unordered
pairs, allowing repetition. We prove that, for every prime $p\ge64$,
\[
 m(p)\ge2^{-80}\left(\frac{\log p}{\log\log p}\right)^2.
\]
The argument compresses the full integer collision lattice by unit-pivot
elimination. A shared random sample and forests of bounded diameter give
$O(\sqrt n+n/\log p)$ surviving coordinates of polynomial height for a
minimal set of size $n$. A nonzero minor divisible by $p$ then gives the
lower bound. We also construct weakly ternary-balanced seeds yielding
\[
 m(p)\le\frac{(\log p)^2}{2(\log3)^2}
 +\left(\frac1{4\log3}+o(1)\right)
 \frac{(\log p)^2}{\log\log p}.
\]
Consequently $m(p)=(\log p)^{2+o(1)}$ as $p$ tends to infinity through
the primes. The constant-factor order of $m(p)$ remains undetermined.
\end{abstract}
\maketitle

\section{Introduction}\label{sec:introduction}
For a set $A\subseteq\F_p$, define
\[
 r_{A,2}(s)=\#\{\{a,b\}:a,b\in A,\ a+b=s\},
\]
where $\{a,b\}$ is an unordered two-element multiset. In particular,
$\{a,b\}$ and $\{b,a\}$ give the same representation, and $\{a,a\}$
is allowed. We say that $A$ has \emph{no unique sums}, or is \emph{NUS},
if $|A|\ge2$ and $r_{A,2}(s)\ge2$ for every $s\in A+A$. For odd primes put
\[
 m(p)=\min\{|A|:A\subseteq\F_p\text{ is NUS}\}.
\]
This minimum exists, since $\F_p$ itself is NUS. Indeed, for
each $s\in\F_p$ there are $p$ ordered solutions of $a+b=s$,
exactly one of which is diagonal. Identifying the two orders of
each nondiagonal pair leaves $(p+1)/2\ge2$ unordered representations.

Problem~27 in Green's list~\cite{Green} asks for the order of magnitude
of $m(p)$. Bedert~\cite{Bedert} established
\begin{equation}\label{eq:bedert}
 \log p\,\frac{\sqrt{\log^{(3)}p}}{\log^{(4)}p}
 \ll m(p)\le\left(\frac12+o(1)\right)(\log_2p)^2,
\end{equation}
where $\log^{(j)}$ denotes $j$ iterates of the natural logarithm.
Cao and Yuan~\cite{CY} subsequently proved the lower bound
$m(p)\gg\log p\log\log p$. Their preprint also gives
\begin{equation}\label{eq:cy-upper}
 m(p)\le\frac{(\log p)^2}{2(\log3)^2}
 +\left(\frac2{\log3}+o(1)\right)
 \frac{(\log p)^2}{\log\log p}.
\end{equation}

Related determinant methods occur in the two-set unique-sum problem.
Leung and Schmidt~\cite{LS} study pairs $A,B$ in finite abelian groups
for which no element of $A+B$ has a unique representation by an element
of $A\times B$. They use sparse linear systems, determinant estimates
and Smith normal forms to obtain structural information about such
pairs. This provides background for the use of integer collision
relations in the present problem.

Our main result is the following estimate.
\begin{theorem}\label{thm:main-lower}
For every prime $p\ge64$ and every NUS set $A\subseteq\F_p$,
\begin{equation}\label{eq:unconditional-lower}
 |A|\ge2^{-80}\left(\frac{\log p}{\log\log p}\right)^2.
\end{equation}
\end{theorem}
The constant is explicit and has not been optimized. Relative to the
scale $\log p\log\log p$, the new lower-bound scale gains the unbounded
factor $\log p/(\log\log p)^3$.

A separate construction sharpens the second term of
\eqref{eq:cy-upper}.
\begin{theorem}\label{thm:main-upper}
As $p$ tends to infinity through the primes,
\begin{equation}\label{eq:ternary-upper}
 m(p)\le\frac{(\log p)^2}{2(\log3)^2}
 +\left(\frac1{4\log3}+o(1)\right)
 \frac{(\log p)^2}{\log\log p}.
\end{equation}
\end{theorem}
The leading coefficient in \eqref{eq:ternary-upper} is the same as in
Cao--Yuan's bound; the second coefficient is reduced by a factor of eight.
Both estimates hold along all sufficiently large primes.

\begin{corollary}\label{cor:exponent}
As $p$ tends to infinity through the primes,
\[
 \frac{\log m(p)}{\log\log p}\longrightarrow2,
 \qquad m(p)=(\log p)^{2+o(1)}.
\]
\end{corollary}
\begin{proof}
Write $X=\log p$, so that $\log X=\log\log p$ tends to infinity.
The lower bound gives, for all sufficiently large primes,
\[
 \log m(p)\ge 2\log X-2\log\log X-80\log2.
\]
To use the upper bound, it is enough to fix one positive value of the
error in Theorem~\ref{thm:main-upper}. Since $\log X\ge1$ eventually,
there is an absolute constant $C_0>0$ such that $m(p)\le C_0X^2$.
Consequently
\[
 2-2\frac{\log\log X}{\log X}-\frac{80\log2}{\log X}
 \le\frac{\log m(p)}{\log X}
 \le2+\frac{\log C_0}{\log X}.
\]
Both error terms on the left and the error term on the right tend to
zero. This proves the first assertion by squeezing.

For the second assertion, define
\[
 \varepsilon(p)=\frac{\log m(p)}{\log\log p}-2.
\]
We have just proved that $\varepsilon(p)\to0$ along the primes. Since
$m(p)>0$ and $\log p>1$ in the range under consideration, exponentiating
the defining identity gives the exact equality
$m(p)=(\log p)^{2+\varepsilon(p)}$. Thus the power notation in the
statement refers to an actual error function tending to zero.
\end{proof}

\subsection{Outline of the lower-bound argument}
Choose an inclusion-minimal NUS set $A$ and put $n=|A|$ and $X=\log p$.
Its full collision lattice $\Lambda\subseteq\Z^A$ has rational rank
$n-1$, while its rank modulo $p$ is at most $n-2$. This rank difference
will supply a nonzero integer determinant divisible by $p$.

The main task is to reduce the number of coordinates while controlling
their size. We maintain a free quotient of $\Z^A$ whose basis consists
of images of actual points of $A$. The moving coordinates represent
disjoint subsets of $A$, each of size less than
$\lfloor(\log_2p)/3\rfloor$. The clean rectangle lemma forces an
escaping collision between every two such subsets. Repeated target
types give unit relations directly. When many target types are distinct,
one common sample of $O(\sqrt m)$ coordinates produces enough unit
relations among the $m$ moving coordinates.

A forest of bounded diameter then eliminates a fixed proportion of
the moving coordinates using pivots $\pm1$. No division or saturation
is needed. Iteration leaves
\[
 K=O\left(\sqrt n+\frac nX\right)
\]
coordinates, and every actual point has coordinate norm at most $n^{18}$.
The compressed rank difference and Hadamard's inequality yield
\[
 X\ll\left(\sqrt n+\frac nX\right)\log n.
\]
In the only relevant range $n<X^2$, this implies
$n\gg(X/\log X)^2$.

There are two separate estimates to keep in view. The dimension bound
controls the number of rows in the eventual determinant; the height
bound controls the size of each row. Neither estimate alone suffices.
The short-forest argument is used for both purposes: having few
components reduces dimension, while having short paths prevents a
single elimination round from increasing the height by a factor
depending on $n$. A fixed height multiplier over $O(\log n)$ rounds
then gives a polynomial height.

Inclusion-minimality enters the rational-rank argument. The combinatorial
compression instead uses the NUS property and the disjointness of the
active blocks. Keeping these roles separate is useful: a quotient is
not itself being asserted to be a new NUS set, and its coordinates
need not be elements of $\F_p$. They are integer coordinates of the
original points, equipped with a label map to $\F_p$.

The two inputs to compression address different obstructions. The
rectangle lemma rules out a pair of small blocks whose repairs remain
entirely inside that pair. It therefore supplies relations with a
different target-block pattern, but does not by itself supply unit pivots.
The sampling argument turns sufficiently many of these relations into
unit relations, each involving at most two active coordinates.
The forest then selects a compatible
subcollection of those differences. Finally, the full rational rank
ensures that the relations not used in compression still contain a
nonzero determinant of the required order. This last use of the full
collision lattice is why it is not necessary to preserve any particular
initially chosen basis of collision rows.

For reference, the principal quantities in the lower-bound proof are
collected below. The distinction between $n$ and $m$ is particularly
important: $n$ is fixed, whereas $m$ decreases during compression.
\[
\begin{array}{c|l}
\text{quantity}&\text{meaning}\\ \hline
n&\text{number of points in the original minimal NUS set}\\
X&\log p\\
\ell_*&\lfloor(\log_2p)/3\rfloor,\ \text{threshold for heavy blocks}\\
m&\text{number of active blocks in the current round}\\
H&\text{current upper bound for every core offset}\\
K&\text{rank of the final free quotient}.
\end{array}
\]
The proof first obtains estimates in the moving variables $m,H$,
then sums them in terms of $n$, and only at the last stage compares
$n$ with $X$. In particular, the construction never assumes the
lower bound that it is intended to prove.

\subsection{Organization and notation}
Section~\ref{sec:collisions} proves the rank and rectangle statements.
Section~\ref{sec:sampling} provides the common samples and the short
forest. The integer compression is carried out in
Section~\ref{sec:compression}, and Section~\ref{sec:lower} completes
the lower-bound proof. The independent upper-bound construction is in
Appendix~\ref{app:upper}. Appendix~\ref{app:elementary} records useful
finite bounds for all odd primes.

All unsubscripted logarithms are natural. Norms of integer vectors are
$\ell^1$ norms unless a subscript specifies otherwise. A subgroup
$P\subseteq\Z^A$ is \emph{primitive} if $\Z^A/P$ is torsion-free.
An \emph{actual point} is an element of the original set $A$; this
terminology distinguishes its image in a quotient from an arbitrary
integer combination of the quotient basis.

\section{Collision lattices and clean rectangles}\label{sec:collisions}
An NUS set is \emph{inclusion-minimal} if it has no proper NUS subset.
Every finite NUS set contains such a subset. Write
\[
 \Z^A_0=\left\{v\in\Z^A:\sum_{a\in A}v_a=0\right\}.
\]
The full collision lattice is generated by all rows
\begin{equation}\label{eq:actual-row}
 e_a+e_b-e_c-e_d,
 \qquad a+b=c+d,\quad \{a,b\}\ne\{c,d\},
\end{equation}
with $a,b,c,d\in A$. It lies in $\Z^A_0$ and is denoted by $\Lambda$.
The maximal-level argument below is the standard minimality argument
also used in the collision-lattice proof in~\cite[Proposition~2.3]{CY}.
We include it in the original point coordinates.

\begin{lemma}[Rational rank and modular defect]\label{lem:complete-rank}
If $A$ is inclusion-minimal NUS, then
\[
 \rk_\Q\Lambda=n-1,
 \qquad \rk_{\F_p}\Lambda\le n-2.
\]
\end{lemma}
\begin{proof}
Let $M$ be the matrix of all actual collision rows, with columns indexed
by $A$. A vector $x=(x_a)_{a\in A}\in\Q^A$ lies in its right kernel
precisely when
\[
 x_a+x_b=x_c+x_d
 \quad\hbox{whenever}\quad a+b=c+d.
\]
For a trivial unordered representation this identity is automatic, so
it makes no difference whether zero collision rows are included in $M$.

Set $q=\max_{a\in A}x_a$ and $B=\{a\in A:x_a=q\}$. This is nonempty.
For any $a,b\in B$, the NUS property supplies a different representation
$a+b=c+d$ in $A$. The kernel identity gives $x_c+x_d=2q$. Each summand
is at most $q$, so both must equal $q$. Hence $c,d\in B$.
We have proved that every pair in $B$ has a different repair entirely
inside $B$, including diagonal pairs.

This also shows that $B$ cannot be a singleton: its only unordered pair
would have no different internal repair. Thus $B$ is itself NUS.
Inclusion-minimality forces $B=A$, and $x$ is constant. Conversely,
constants annihilate every row because its coefficients sum to zero.
The rational right kernel therefore has dimension one; rank--nullity
gives $\rk_\Q M=n-1$.

Equivalently, the rows of $M$ test all functions on $A$ that preserve
its additive collisions. Over $\Q$, minimality says that every such
function is constant. The order on $\Q$ is essential to this conclusion:
from $x_c,x_d\le q$ and $x_c+x_d=2q$ we can force both terms to be
$q$. Over $\F_p$ there is no corresponding maximum-layer argument,
and the original labels themselves give a nonconstant
collision-preserving function. The rational and modular conclusions
below are therefore compatible. The proof exploits this difference
between the two coefficient fields.

This argument uses all collision identities, rather than a prescribed
repair for each pair. When a pair belongs to the maximum layer, its
NUS repair is one of the identities annihilated by $x$, so that repair
stays in the layer. Also, minimality is applied only after checking
both requirements in the definition of NUS: the layer has at least
two points and every unordered pair in it has a different internal
representation. Thus no claim about the rank of a general, nonminimal
NUS set is needed.

Over $\F_p$, the coefficient-sum identity still shows that $\one$ is
in the right kernel. The actual-label vector $\lambda=(a)_{a\in A}$
is also in that kernel, now because the original equality $a+b=c+d$
holds in $\F_p$. If $s\one+t\lambda=0$ with $t\ne0$, every label
would equal $-s/t$, contrary to the distinctness of the elements of $A$.
If $t=0$, then $s=0$. These two vectors are independent, so the modular
right kernel has dimension at least two and the modular rank is at
most $n-2$.
\end{proof}

The next lemma is the total unimodularity property of directed
vertex--edge incidence matrices, with the matrix transposed; see
Schrijver~\cite[Chapter~19]{Schrijver}. We record its elementary
determinant proof in the row convention used here.

\begin{lemma}[Unit-incidence minors]\label{lem:incidence-minors}
Every square minor of a matrix whose rows are differences of standard basis vectors belongs to $\{0,1,-1\}$.
\end{lemma}
\begin{proof}
A restriction of a row $e_u-e_v$ to selected columns has either no
nonzero entries, one entry equal to $1$ or $-1$, or two entries equal
to $1$ and $-1$. In particular the assertion remains meaningful when
one or both endpoints lie outside the selected column set, or $u=v$.

Induct on the order $k$ of the selected square submatrix. For $k=0$
the determinant is $1$. A zero row gives determinant zero. If a row
has exactly one nonzero entry, expansion along that row expresses the
determinant as $1$ or $-1$ times a determinant of order $k-1$; the
remaining rows are still restrictions of unit differences, so induction
applies. In the only remaining case every row has two nonzero entries.
Its row sum is then zero, and the all-ones column vector is a nonzero
kernel vector. The determinant is zero in this case as well.
\end{proof}

\begin{lemma}[Clean rectangle]\label{lem:rectangle}
Let $I,J\subset\F_p$ be nonempty and disjoint. Suppose every pair in $I\times J$ has a different representation inside the same rectangle, with one endpoint in each side. Then
\[
 p\le\binom{|I|+|J|-2}{|I|-1}\le2^{|I|+|J|-2}.
\]
\end{lemma}
\begin{proof}
Call a nonempty subrectangle \emph{repair-closed} if every pair in it
has a different representation with both endpoints still on the
appropriate sides. Among repair-closed subrectangles of $I\times J$,
choose $I'\times J'$ minimizing $|I'|+|J'|$, and put $x=|I'|$,
$y=|J'|$. If one side were a singleton, cancellation would make the
sum map on the rectangle injective. Hence $x,y\ge2$.

\medskip\noindent\emph{The rational kernel.}
Use a separate coordinate for each point on each side. An actual row
is $(e_a-e_c)\oplus(e_b-e_d)$, where $a,c\in I'$, $b,d\in J'$,
and $a+b=c+d$. The two vectors that are constant on one side and zero
on the other annihilate every row. Suppose $(f,g)\in\Q^{I'}\oplus
\Q^{J'}$ is any rational right-kernel vector. Let $I''$ and $J''$
be the maximum layers of $f$ and $g$. For $a\in I''$, $b\in J''$,
any repair satisfies
\[
 f(c)+g(d)=f(a)+g(b)=\max f+\max g.
\]
Neither term can exceed its own maximum, so $c\in I''$ and $d\in J''$.
Thus $I''\times J''$ is repair-closed. Minimality implies $I''=I'$
and $J''=J'$. The kernel consists exactly of the two side constants,
and the rational row rank is $x+y-2$.

\medskip\noindent\emph{An integer minor divisible by $p$.}
Fix anchors $a_0\in I'$, $b_0\in J'$ and delete their columns.
Every row has sum zero separately on each side, so its deleted entries
are determined by its retained entries. Deletion is therefore injective
on the rational row space and preserves its dimension. Select $x+y-2$
actual rows that remain independent. They form a square integer matrix
$D$ with $\det D\ne0$.

The modular vector on the retained columns given by
\[
 (a-a_0)_{a\in I'\setminus\{a_0\}}
 \oplus(b-b_0)_{b\in J'\setminus\{b_0\}}
\]
is annihilated by $D$: subtracting a side constant does not change a
row pairing, and the actual labels annihilate every collision. This
vector is nonzero because $x,y\ge2$ and the labels on each side are
distinct. Hence $D$ is singular modulo $p$, so $p\mid\det D$ and
$p\le|\det D|$.

The same integer matrix is being used in these two fields. Its
determinant is a nonzero integer by rational independence, and its
reduction modulo $p$ is zero by the exhibited modular kernel vector.
Reduction commutes with the determinant polynomial, so these facts
give divisibility by $p$. They do not require choosing a common
basis for the rational and modular kernels.

\medskip\noindent\emph{The determinant bound.}
Expand $\det D$ along the $x-1$ retained columns of the first side.
A term is indexed by a choice of $x-1$ rows and, up to sign, is a
product of two square incidence minors, one on each side. Each factor
has absolute value at most one by Lemma~\ref{lem:incidence-minors}.
There are $\binom{x+y-2}{x-1}$ choices, so
\[
 p\le|\det D|\le\binom{x+y-2}{x-1}.
\]
The quantity $\binom{r+s}{r}$ is nondecreasing separately in the
nonnegative integers $r,s$; this follows directly from Pascal's identity.
Since $x\le|I|$ and $y\le|J|$, it is at most the binomial coefficient
in the statement. Finally any one binomial coefficient in row
$|I|+|J|-2$ of Pascal's triangle is at most the sum of that row,
namely $2^{|I|+|J|-2}$.

The Laplace expansion above can be written explicitly. Write
$D=[U\ V]$, where $U$ has $x-1$ columns and $V$ has $y-1$ columns,
and put $k=x+y-2$. Then
\[
 \det D=\sum_{\substack{S\subseteq\{1,\ldots,k\}\\|S|=x-1}}
       \varepsilon_S\det U_{S,*}\det V_{S^c,*},
 \qquad \varepsilon_S\in\{1,-1\}.
\]
The row of $U$ associated with a collision is the restriction of
$e_a-e_c$, and the corresponding row of $V$ is the restriction of
$e_b-e_d$. Thus the incidence-minor lemma applies separately to both
factors in each term. The original rows of $D$ can have four nonzero
entries; it is their decomposition into two incidence blocks that
gives this sharper bound. Taking absolute values removes the signs
$\varepsilon_S$, and the number of row subsets is the only remaining
factor.
\end{proof}

We shall use the rectangle lemma in its contrapositive form. If
$2^{|I|+|J|-2}<p$, there is at least one pair $(a,b)\in I\times J$
with no different representation in that rectangle. If $I,J$ lie
inside an NUS set, that particular pair still has a repair in the
whole set, so every such repair leaves the rectangle. Only one
escaping record per pair of blocks is needed. No assertion that
every cross pair escapes is required.

\section{Shared samples and short forests}\label{sec:sampling}
The first lemma selects a single set that works for many families at
once. Within a family both the first labels and the target multisets
must be distinct; repeated target types will be handled separately in
the compression argument.
\begin{lemma}[Simultaneous configuration sample]\label{lem:sample}
Let $V$ have $m\ge2^{32}$ elements. A configuration is $(j,\tau)$, where $\tau$ is a two-element multiset in $V$ and $j\notin\supp\tau$. Suppose there are at most $m$ families, each containing at least $m/8$ configurations, with distinct $j$'s and distinct $\tau$'s within each family. There is a common $T\subseteq V$ such that
\[
 |T|\le4096\sqrt m,
\]
and at most $m/8$ families fail to contain a configuration with $j\in T$ and $\supp\tau\cap T\ne\varnothing$.
\end{lemma}
\begin{proof}
We use two independent samples so that activating a target and hitting
that target are separate random operations. Choose $X,Y\subseteq V$
independently, with each vertex included in either sample independently
with probability
\[
 \theta=\frac{512}{\sqrt m},
 \qquad T=X\cup Y.
\]
The hypothesis $m\ge2^{32}$ ensures $0<\theta<1$. The same pair
$(X,Y)$ will be used for every family.

For a single configuration, requiring both its label and a target
vertex to be sampled suggests a probability of order $\theta^2$.
However, many configurations can share a target vertex, so treating
their success events as independent would be unjustified. The random
union $U(X)$ below collects all activated targets first. Conditional
on that union, independence of the second sample gives an exact
probability of missing it. This is the reason for using two samples.

\medskip\noindent\emph{Coverage for one family.}
Fix a family $\mathcal F$ with $a$ configurations, and write
$S_j=\supp\tau_j$ for the support of the type paired with label $j$.
The labels are distinct, so their inclusion indicators in $X$ are
independent. Each $S_j$ is nonempty and has size at most two.
Moreover the supports are distinct: a two-element multiset with
singleton support is necessarily $\{v,v\}$, and one with two-element
support is necessarily $\{v,w\}$. Thus its support determines the type.

Let
\[
 U(X)=\bigcup_{j\in X,\ (j,\tau_j)\in\mathcal F}S_j,
 \qquad Z=|U(X)|,
 \qquad d_v=\#\{j:v\in S_j\}.
\]
A vertex $v$ is uncovered exactly when all its $d_v$ activating labels
are absent from $X$. Linearity of expectation, without any independence
assumption between coverage events for different vertices, gives
\begin{equation}\label{eq:coverage-mean-detailed}
 \mu:=\mathbb EZ
 =\sum_{v\in V}\bigl(1-(1-\theta)^{d_v}\bigr)
 \ge\frac12\sum_{v\in V}\min\{1,\theta d_v\}.
\end{equation}
Here $(1-\theta)^d\le\ee^{-\theta d}$ and
$1-\ee^{-s}\ge\tfrac12\min\{1,s\}$ for $s\ge0$.

To justify the latter elementary estimate, for $0\le s\le1$ use
$\ee^{-s}\le1-s+s^2/2$, which gives
$1-\ee^{-s}\ge s-s^2/2\ge s/2$. For $s\ge1$, monotonicity
gives $1-\ee^{-s}\ge1-\ee^{-1}>1/2$. Thus the minimum records
two regimes: a vertex with small $\theta d_v$ is covered with
probability proportional to $\theta d_v$, whereas a vertex with
large $\theta d_v$ already has coverage probability bounded below
by an absolute constant.

Put $H=\{v:\theta d_v\ge1\}$. If $|H|\ge\sqrt m/16$, the
contribution of $H$ to \eqref{eq:coverage-mean-detailed} gives
$\mu\ge\sqrt m/32$. Otherwise, the number of distinct nonempty
supports of size at most two contained in $H$ is at most $|H|^2$:
each such support is the image of an ordered pair in $H\times H$,
with equal coordinates allowed. Thus fewer than $m/256$ of our
types are contained in $H$. Since $a\ge m/8\ge m/64$, at least
$m/128$ types have a vertex outside $H$. Counting one such incidence
for each of these types yields
\[
 \sum_{v\notin H}d_v\ge\frac m{128},
 \qquad
 \mu\ge\frac\theta2\sum_{v\notin H}d_v
 \ge2\sqrt m.
\]
In both cases,
\begin{equation}\label{eq:coverage-mean-lower-detailed}
 \mu\ge\frac{\sqrt m}{32}.
\end{equation}

This dichotomy explains the square-root sampling scale. If many
vertices have high degree, their individual contributions already
give the required coverage. If few do, distinctness of the target
supports prevents all the configurations from lying among them:
only quadratically many different supports fit on a given set of
vertices. Many incidences must then fall on low-degree vertices,
where their contributions add linearly. The argument consequently
handles both heavily overlapping targets and widely spread targets
with the same choice of $\theta$.

For a small illustration of the dependence being handled, consider
two labels with target supports $\{u,v\}$ and $\{u,w\}$, where
$u,v,w$ are distinct. If both labels are activated, their union has
size three, not four. Conditional on those activations, the second
sample misses both targets with probability $(1-\theta)^3$.
Multiplying the two individual miss probabilities would instead
give $(1-\theta)^4$, which is incorrect because the targets share
$u$. The union calculation retains the correct dependence. In the
variance argument below, deleting either of the two activated labels
loses only its private vertex, $v$ or $w$; the shared vertex $u$
survives both individual deletions. This is the simplest example
of the disjoint-loss counting used for a general family.

\medskip\noindent\emph{Why the coverage has small variance.}
For a label $j$, force its activation indicator to zero and denote
the resulting loss of covered vertices by $\Delta_j$. At most the
two vertices of $S_j$ can be lost, so $0\le\Delta_j\le2$.
A vertex contributes to a loss only if $j$ is its unique selected
covering label. Such a vertex contributes to at most one of the losses.
Consequently, pointwise,
\[
 \sum_j\Delta_j\le Z,
 \qquad \sum_j\Delta_j^2\le2Z.
\]
The variance tensorization inequality for independent coordinates says
that $\operatorname{Var}Z$ is at most the sum of the expected
conditional variances in the individual coordinates. This is the
conditional-variance form of the Efron--Stein inequality; see
Boucheron, Lugosi and Massart~\cite[Section~3.1]{BLM}. The deletion-loss
estimate above is also an instance of the self-bounding method
discussed in~\cite[Section~3.3]{BLM}. To see the bound needed here,
fix all indicators except the one at $j$. If
changing that indicator from zero to one increases coverage by $d$,
the conditional variance is $\theta(1-\theta)d^2$, whereas the
conditional expectation of $\Delta_j^2$ is $\theta d^2$.
Tensorization therefore gives
\[
 \operatorname{Var}Z\le\sum_j\mathbb E\Delta_j^2\le2\mu.
\]
For completeness, tensorization here can be seen directly by ordering
the activation indicators $\xi_1,\ldots,\xi_a$ and setting
$M_i=\mathbb E[Z\mid\xi_1,\ldots,\xi_i]$. The increments
$M_i-M_{i-1}$ are orthogonal in $L^2$, so
$\operatorname{Var}Z=\sum_i\mathbb E(M_i-M_{i-1})^2$.
With the earlier indicators fixed, let $d_i$ be the change in $Z$
when the $i$th indicator is changed from zero to one, as a function
of the later indicators. Independence gives
\[
 M_i-M_{i-1}=(\xi_i-\theta)\mathbb E_{\mathrm{later}}d_i.
\]
Its conditional second moment is
$\theta(1-\theta)(\mathbb E_{\mathrm{later}}d_i)^2$, which is at
most $\theta(1-\theta)\mathbb E_{\mathrm{later}}d_i^2$ by Jensen's
inequality. Averaging and summing proves the tensorization bound
used above. No independence between the covered vertices is required.

One can also see the pointwise loss estimate by assigning vertices
to labels. A covered vertex with two or more selected covering labels
disappears under none of the single-label deletions. A covered vertex
with exactly one selected covering label is assigned to that label,
and disappears only when that label is deleted. Hence the sets counted
by the positive losses $\Delta_j$ are disjoint subsets of $U(X)$.
The bound $\Delta_j\le2$ then gives
$\Delta_j^2\le2\Delta_j$ before taking any expectations.

Chebyshev's inequality and \eqref{eq:coverage-mean-lower-detailed} imply
\[
 \Pr(Z<\mu/2)\le\frac{4\operatorname{Var}Z}{\mu^2}
 \le\frac8\mu\le\frac{256}{\sqrt m}.
\]
Condition on $X$. The independent sample $Y$ misses $U(X)$ with
probability exactly $(1-\theta)^Z$. On $Z\ge\mu/2$ we have
$\theta Z\ge8$, so this probability is at most $\ee^{-8}$.
If $Y$ hits $U(X)$, some activated label $j$ and a vertex of its
target support both belong to $T$. Thus the failure probability for
this family is at most
\[
 \frac{256}{\sqrt m}+\ee^{-8}
 \le\frac1{256}+\ee^{-8}<\frac1{128}.
\]
This argument also covers a repeated target $\{v,v\}$, whose support
is the singleton $\{v\}$.

\medskip\noindent\emph{One sample for all families.}
Let $F$ be the number of failed families for the common set $T$.
There are at most $m$ families, so $\mathbb EF\le m/128$.
The failures may be dependent; only linearity of expectation is used.
Also
\[
 \mathbb E|T|\le\mathbb E|X|+\mathbb E|Y|
 =2\theta m=1024\sqrt m.
\]
Markov's inequality bounds the probabilities of $F>m/8$ and
$|T|>4096\sqrt m$ by $1/16$ and $1/4$, respectively. Their sum is
less than one, so a single realization satisfies both desired bounds.

Explicitly, if $E_1=\{F>m/8\}$ and
$E_2=\{|T|>4096\sqrt m\}$, the union bound gives
$\Pr(E_1\cup E_2)\le1/16+1/4<1$. A realization outside this
union works simultaneously for every successful family. We never
multiply probabilities belonging to different families, and we never
choose a separate realization of $T$ for each family. Finally,
enlarging $T$ preserves success, since success is an existence
condition on two memberships. This monotonicity permits its later
union with the easy-label sample.
\end{proof}

\begin{lemma}[Easy-label sample]\label{lem:easy}
With $|V|=m\ge2^{32}$ as in the preceding lemma, for at most $m$ subsets of $V$, each of size at least $m/8$, there is a common hitting set of size at most $2048\sqrt m$.
\end{lemma}
\begin{proof}
Select each vertex independently with probability
$\theta=1024/\sqrt m$. For a fixed set $E$ of size at least $m/8$,
\[
 \Pr(T\cap E=\varnothing)=(1-\theta)^{|E|}
 \le\exp(-\theta|E|)\le\exp(-128\sqrt m).
\]
The union bound over at most $m$ sets shows that the probability of
missing any prescribed set is at most $m\exp(-128\sqrt m)<1/4$.
For example, writing $s=\sqrt m\ge2^{16}$ and using
$\log m=2\log s\le2s$, this expression is at most $\ee^{-126s}$.

The expected size of the sample is $\theta m=1024\sqrt m$.
By Markov's inequality, its size exceeds $2048\sqrt m$ with
probability at most $1/2$. The total probability of either failure
is less than one. Therefore some sample has the required size and
hits every prescribed set simultaneously.
\end{proof}

The distinction between the two sampling lemmas will matter later.
An easy family needs only its partner label to be promoted. A genuine
family needs both a partner label and at least one target label to be
promoted. The first lemma therefore allows a small number of missed
families, whereas the easy-label lemma hits every prescribed set.
Both costs are $O(\sqrt m)$, and their union can be promoted in a
single compression round.

\begin{lemma}[Short forest]\label{lem:short-forest}
In a graph on $N$ active vertices and an optional root $0$, suppose at most $f$ active vertices have neither an active neighbour nor a root edge. There is a forest whose unrooted components have diameter at most three, whose rooted vertices have root distance at most four, and whose number of unrooted components is at most $(N+f)/2$.
\end{lemma}
\begin{proof}
First ignore the optional root and take an inclusion-maximal matching
among the active vertices. No two unmatched vertices can be adjacent,
since their edge could otherwise be added to the matching. Hence every
unmatched vertex that has an active neighbour has a matched neighbour.
Attach it to one chosen matched neighbour by a single edge.

The resulting components containing matching edges are double stars:
each consists of one matched edge, with leaves attached to either
endpoint. They are trees, and a path between any two vertices uses at
most a leaf edge, the matched edge, and another leaf edge. Their
diameter is at most three. The remaining active vertices have no
active neighbour and are singleton components at this stage.

For each double star containing a vertex adjacent to the root, choose
one such vertex and one root edge, and add only that root edge. Also
join each active singleton with a root edge directly to the root.
Adding a single root edge to each of several disjoint trees cannot
create a cycle. Any vertex in a rooted double star is at distance at
most three from its chosen attachment vertex and hence at most four
from the root. Rooted singletons have distance one.

Let $q$ be the number of components left unrooted, and let $s$ of them
be singletons. Each such singleton has neither an active neighbour nor
a root edge, so $s\le f$. Each other unrooted component contains at
least two active vertices. Counting only the vertices in unrooted
components gives
\[
 N\ge s+2(q-s)=2q-s\ge2q-f.
\]
Thus $q\le(N+f)/2$. The construction also applies when the root is
absent, when the matching is empty, or when there are no active vertices.
\end{proof}

The diameter condition is as important as the component count. Along
a path, the core corrections in successive unit relations add. A
general spanning tree could have a path of length comparable with
$N$, introducing an $N$-dependent loss at every round. The double-star
construction bounds that accumulation by four edge corrections,
including the possible final edge to the root.

\section{Primitive compression with actual anchors}\label{sec:compression}
An anchor is an actual point whose image is one of the current basis
vectors. Keeping such anchors will preserve the degree and actual-label
maps throughout the argument.

\begin{theorem}\label{thm:compression}
Let $p\ge64$ be prime and let $A\subseteq\F_p$ be inclusion-minimal NUS,
with $n=|A|$. There is a subgroup $P\subseteq\Lambda$, primitive in
$\Z^A$, such that
\[
 \Z^A/P\simeq\Z^K,
 \qquad K\le2^{33}\left(\sqrt n+\frac n{\log p}\right).
\]
There is a basis consisting of images of distinct actual points in which
every $a\in A$ has coordinates $z_a$ satisfying
\[
 \sum_i(z_a)_i=1,\qquad \norm{z_a}_1\le n^{18}.
\]
The subgroup $P$ is obtained by successively adding genuine collision
rows or differences of two such rows, with every elimination pivot
equal to $1$ or $-1$.
\end{theorem}
\begin{proof}
Put $\ell_* =\lfloor(\log_2p)/3\rfloor\ge2$.
The purpose of the construction is to identify many coordinates while
keeping a free integer quotient. We retain actual points as basis
anchors, so that the quotient also remembers the original labels in
$\F_p$. The quantitative issue is to do this with few surviving
coordinates and without making the coordinates of the other points
too large.

The proof has three bookkeeping tasks. First, every added relation
must remain in the original lattice $\Lambda$. Second, the quotient
must have an explicit integer basis, rather than only a rational
basis. Third, the number of surviving coordinates and the size of
all point coordinates must be estimated simultaneously. The first
four steps construct and preserve the state; the final step sums
the costs over all rounds.

\medskip\noindent\emph{Step 1: the state and its invariants.}
At any stage we have a subgroup $P\subseteq\Lambda$ and an explicit
free basis of $\Z^A/P$. The basis is divided into permanent core
coordinates and active coordinates $x_B$. The latter are indexed by
pairwise disjoint physical blocks $B\subseteq A$, each of size less
than $\ell_*$. For $a$ in an active block, and for an inactive point,
respectively, the coordinate descriptions are
\[
 \overline e_a=x_B+g_a,\qquad \overline e_a=g_a.
\]
Every $g_a$ uses only permanent core coordinates, and one common bound
$H\ge1$ satisfies $\norm{g_a}_1\le H$. Each active block has an actual
anchor $a_B\in B$ with $g_{a_B}=0$. Every permanent core is likewise
the image of an actual point. Initially $P=0$, the blocks are the
singletons of $A$, there are no cores, and $H=1$ is valid because all
offsets are zero.

More explicitly, if $\mathcal C$ is the current set of cores and
$\mathcal B$ is the current collection of active blocks, the state
specifies an isomorphism
\[
 \Z^A/P\simeq\Z^{\mathcal C}\oplus\Z^{\mathcal B}.
\]
An active point has exactly one active coordinate, with coefficient
one. All other information about that point is in its core offset.
Inactive points have no active coordinate. Thus physical blocks
describe disjoint subsets of the original set, while basis vectors
describe the current free quotient; these are related but different
objects. In particular, a merger changes the partition into blocks,
whereas promotion changes which part of the basis is called active.

To \emph{promote} a block means to reclassify its basis coordinate as a
permanent core and make its points inactive. Promotion adds no relation
and changes no quotient: it only changes the division of its basis.
A point formerly written $x_B+g_a$ now has a core-only expression of
norm at most $H+1$. The old anchor becomes the new core anchor.
Whenever a merged block has size at least $\ell_*$, we promote it and
call it heavy. An inactive point never becomes active again. This
last rule will allow us to charge heavy promotions to disjoint subsets
of the original set.

Two homomorphisms accompany every state. The degree map to $\Z$ and
the label map to $\F_p$ are induced by
\[
 e_a\longmapsto1,\qquad e_a\longmapsto a,
\]
respectively. Both annihilate $P$ because $P\subseteq\Lambda$.
For any chosen anchor basis $(a_i)$, they consequently satisfy
\[
 \sum_i(z_a)_i=1,\qquad
 \sum_i(z_a)_i a_i=a\quad\hbox{in }\F_p.
\]
The second identity explains why actual anchors will later supply
a modular kernel vector. It does not require the integer coordinates
$(z_a)_i$ to be positive or to be small modulo $p$.

The degree map also describes the offsets. For an active point,
$\deg(g_a)=\deg(\overline e_a)-\deg(x_B)=0$; for an inactive
point, $\deg(g_a)=1$. Here the degree of a core expression is the
sum of its coefficients, because each core is an actual anchor.
Consequently a unit row $x_i-x_j+h$ has $\deg(h)=0$, whereas
a root row $x_i+h$ has $\deg(h)=-1$. Solving the latter as
$x_i=-h$ gives a degree-one core expression, as it must. In
particular, the formal root has value zero in the active-coordinate
notation; it is not an actual anchor of degree one.

\medskip\noindent\emph{Step 2: escaping collisions and unit relations.}
Suppose there are $m\ge2^{32}$ active blocks. If every cross pair of
two different blocks $B_i,B_j$ had a different representation within
$B_i\times B_j$, Lemma~\ref{lem:rectangle} would imply
\[
 p\le2^{|B_i|+|B_j|-2}<p.
\]
Indeed $|B_i|,|B_j|<\ell_*$ and
$2\ell_*\le(2/3)\log_2p$. Thus we may choose a cross pair with no
internal repair. The NUS property still supplies a repair in $A$,
which must escape that rectangle. Its actual collision row has image
\begin{equation}\label{eq:escaping}
 r_{ij}=x_i+x_j-x(\nu_{ij})+g_{ij},
 \qquad \norm{g_{ij}}_1\le4H.
\end{equation}
Here $\nu_{ij}$ is the multiset of active blocks containing the two
target points, with inactive target points omitted; multiplicities
are retained. Thus $|\nu_{ij}|\le2$ and
$x(\nu)=\sum_{k\in\nu}x_k$. Moreover $\nu_{ij}\ne\{i,j\}$:
otherwise the target would be an internal repair of the chosen pair.
Make and retain one such record for each unordered pair of blocks.
The bound $4H$ counts the offsets of the two source and two target
points, including repeated points if present.

A \emph{unit relation} is a row whose active part is $x_i-x_j$ or
$x_i$, up to sign. We represent it by an edge $ij$, or by an edge
from $i$ to a formal root $0$, respectively. The root represents
the absence of an active coordinate; it is not an additional basis
coordinate. Every graph edge keeps the actual integer row that it
represents.

If $\nu_{ij}$ contains $j$, cancellation in \eqref{eq:escaping}
gives a unit relation incident to $i$. Its other target, if present,
cannot be $i$, since $\nu_{ij}\ne\{i,j\}$. This observation also
covers the repeated target $\{j,j\}$. Call $i$ \emph{direct} if such
a record exists. A target containing $i$ similarly gives a unit
relation incident to its partner $j$.

For example, target types $\{j,k\}$, $\{j,j\}$ and $\{j\}$
give active parts $x_i-x_k$, $x_i-x_j$ and $x_i$, respectively.
An omitted target occurrence is an inactive point, whose entire
image has already been included in the core term. These examples
also show why multiplicities in target types cannot be discarded:
the type $\{j,j\}$ contributes $2x_j$ before cancellation.

Now fix a nondirect $i$. No target of its records contains its
partner $j$. Set aside partners whose targets contain $i$, since
they already have incident unit relations. All remaining targets
avoid both $i$ and $j$. Among these, a partner is \emph{easy} if its
target has size zero or one; a target of size two is called a
\emph{genuine type}. Group the latter partners by their exact target
multisets. In a class choose a representative $j_0$. For every other
partner $j$, subtract the two retained records to obtain
\[
 r_{ij}-r_{ij_0}
 =x_j-x_{j_0}+g_{ij}-g_{ij_0},
 \qquad \norm{g_{ij}-g_{ij_0}}_1\le8H.
\]
These are star edges on the partners of the class. In particular,
every vertex of a class of size at least two has an incident edge.
Build the unit graph using these stars and the unit relations already
obtained directly. All its source rows belong to $\Lambda$; a star
edge uses a difference of two actual collision rows.

It is essential here that the two records are compared in the same
current basis. Their identical target multisets then cancel as
integer vectors, even when the actual target points within a block
are different: the differences of those points are core offsets.
The resulting edge is recorded together with its preimage
$r_{ij}-r_{ij_0}$ in the original lattice, not merely with a formal
equality among active labels.

\medskip\noindent\emph{Step 3: obtaining a fixed proportional reduction.}
There are two cases, according to whether any nondirect vertex has
both few easy partners and few genuine types.

First suppose that some nondirect $i$ has fewer than $m/8$ easy
partners and fewer than $m/8$ genuine types. For a partner $j$, a
target containing $i$ already supplies an edge at $j$; a target
containing $j$ is impossible by nondirectness. A partner in a
genuine class of size at least two has a star edge. Consequently
the only possible isolated vertices are $i$, its easy partners,
and the partners in singleton genuine classes. Their number is at
most $m/4+1$. By Lemma~\ref{lem:short-forest}, a selected forest has
at most
\[
 \frac{m+(m/4+1)}2=\frac{5m}8+\frac12\le\frac{2m}3
\]
unrooted components; the last inequality holds for $m\ge12$.
This case requires no sample promotions.

Otherwise every nondirect vertex has at least $m/8$ easy partners
or at least $m/8$ genuine types. Assign each such vertex to one
alternative that it satisfies. Apply Lemma~\ref{lem:easy} to all
assigned easy-partner sets. For a vertex assigned to genuine types,
choose one partner per type and use the configurations
$(j,\nu_{ij})$. In each family the partners are distinct, the types
are distinct, and $j\notin\supp\nu_{ij}$, so Lemma~\ref{lem:sample}
applies. There are at most $m$ families of either kind. Promote the
union of the two resulting samples. This uses at most
\[
 (2048+4096)\sqrt m=6144\sqrt m
\]
new cores. The two existence lemmas can be applied separately;
the union retains every hitting property.

Consider a vertex $i$ that remains active. If its easy family is
hit, an appropriate partner is now a core, and the corresponding
row has active part $x_i$ or $x_i-x_k$. If its genuine-type family
is hit, the partner and at least one target block have been
promoted, with the same conclusion. When the target is $\{k,k\}$,
promoting $k$ removes both active occurrences. Existing unit
relations with at least one remaining active endpoint also remain
unit relations, becoming root edges when necessary. Thus only
vertices assigned to missed genuine-type families can be isolated,
and there are at most $m/8$ of these. If $N\le m$ active vertices
remain, the short-forest lemma leaves at most
\[
 \frac{N+m/8}2\le\frac{9m}{16}\le\frac{2m}3
\]
unrooted components. Its diameter and root-distance conclusions
hold in both cases.

For completeness, the active-coordinate calculation after promotion
can be listed explicitly. In the following table $i$ remains active,
the sampled partner $j$ has become a core, and all terms from promoted
coordinates have been absorbed into the core correction:
\[
 \begin{array}{c|c|c}
 \text{original target type}&\text{additional promotion}
     &\text{remaining active part}\\ \hline
 \varnothing&\text{none}&x_i\\
 \{k\}&k\text{ remains active}&x_i-x_k\\
 \{k\}&k\text{ promoted}&x_i\\
 \{k,\ell\},\ k\ne\ell&k\text{ promoted, }\ell\text{ active}
     &x_i-x_\ell\\
 \{k,\ell\}&k,\ell\text{ promoted}&x_i\\
 \{k,k\}&k\text{ promoted}&x_i.
 \end{array}
\]
The families being sampled have targets avoiding both $i$ and $j$.
Thus a surviving target in this table cannot be $i$, and every
displayed active part is a nonzero unit relation incident to $i$.
If $i$ itself is promoted, it no longer belongs to the active graph
and needs no incident edge. This accounts for both the successful
families and the reduction from $m$ vertices to $N$ vertices.

These alternatives exhaust all states. In the first, repeated
targets already give enough edges to merge many active coordinates.
In the second, there is sufficient variety for the common sample
to reveal enough edges. The sample is shared across all families,
so its $6144\sqrt m$ cost is paid once for the round. Paying that
cost separately for each active vertex would not give the required
dimension estimate.

It is helpful to follow one successful genuine configuration through
this step. Suppose its active part is
$x_i+x_j-x_k-x_\ell$, with four distinct indices. Before promotion
this row alone supplies no unit pivot of the required graph form.
When the common sample contains $j$ and $k$, these two coordinates
become cores, say $c_j,c_k$, and the same row is now
\[
 x_i-x_\ell+(g_{ij}+c_j-c_k).
\]
The expression in parentheses uses only core coordinates, so this
is an edge from $i$ to $\ell$. If $\ell$ was sampled as well, the
row becomes a root edge at $i$. Thus success needs one sampled
partner and one sampled target, rather than all four indices.
The common sample is chosen before selecting the forest; the forest
uses only relations that this one promotion has actually made
available. No further sample depending on an individual edge is
introduced during elimination.

\medskip\noindent\emph{Step 4: why the elimination is valid over $\Z$.}
Add the source rows of the selected forest to the current subgroup
of relations. In an unrooted component choose one coordinate $x_r$
to retain and orient its edges towards $r$. In a rooted component
orient edges towards the formal root and retain no active coordinate.
After changing an edge sign if necessary, its relation has the form
\[
 x_v-x_w+h_{vw}=0,\qquad\hbox{or}\qquad x_v+h_{v0}=0,
\]
where $h$ uses only cores. Starting with leaves, solve for $x_v$
with coefficient $1$. No division or multiplication of a relation
by a nonunit is needed.

The number of relations added in this step can be counted exactly.
Let $N$ be the active count after sample promotion and let $q$ be
the number of unrooted components of the selected forest, including
isolated active vertices. An unrooted tree with $v$ active vertices
has $v-1$ selected edges. A tree attached to the formal root with
$v$ active vertices has $v$ selected edges, counting its root edge.
Summing over the forest gives $N-q$ rows. The unit pivots below
show that these rows are independent in the current quotient.
Thus they remove precisely $N-q$ coordinates and leave precisely
$q$ active coordinates before heavy promotion. This count also
explains why a rooted component incurs no new anchor cost.

For a concrete path $v,w,r$ in an unrooted component, suppose the
two edge rows are $x_v-x_w+h_1$ and $x_w-x_r+h_2$. Solving gives
\[
 x_w=x_r-h_2,\qquad x_v=x_r-h_2-h_1.
\]
A point originally written $x_v+g_a$ therefore becomes
$x_r+(g_a-h_1-h_2)$. The coefficient of the retained active
coordinate is still one, and the old anchor at $r$ still has
offset zero. In a rooted path the same substitution ends at a
core-only expression. The general forest calculation is this
substitution along each unique path.

For clarity, this also identifies the entire relation kernel, not
just its rank over $\Q$. Substitute the solved expressions along
the unique paths towards $r$ or $0$. They express every eliminated
coordinate as $x_r$ plus a core combination, or as a core combination
in the rooted case. They impose no equation on $x_r$ or on a core.
Conversely, subtracting integer multiples of the selected edge rows
reduces any vector to this unique expression in the retained basis.
A vector reduces to zero exactly when it belongs to the integer
span of those edge rows. The new quotient is therefore free over
$\Z$ on the old cores and retained active coordinates.

Taking the preimage in $\Z^A$ gives precisely the old $P$ plus the
chosen source rows. Since those rows lie in $\Lambda$, the new
$P$ is still contained in $\Lambda$. Since its quotient is free,
$P$ is primitive. This proves primitivity directly, without replacing
the constructed subgroup by its saturation.

An explicit quotient map makes the kernel assertion transparent.
Let $E$ be the current free group on cores and active coordinates,
and let $F$ be the free group on the cores and one chosen anchor for
each unrooted component. Keep each core fixed. If $v$ lies in an
unrooted component with chosen anchor $r$, define
\[
 \pi(x_v)=x_r-\sum_{e\in[v,r]}h_e;
\]
if $v$ lies in a rooted component, define
$\pi(x_v)=-\sum_{e\in[v,0]}h_e$. Orient the edge rows along
these paths, so an edge from $v$ to its parent $w$ is
$x_v-x_w+h_e$, with $x_0=0$. The sums telescope and show that
every selected row belongs to $\ker\pi$. Conversely, using these
rows to replace eliminated coordinates reduces any $z\in E$ to
the expression $\pi(z)$ in the retained coordinates. Thus if
$\pi(z)=0$, the original $z$ is an integer combination of the
selected rows. We obtain
\[
 \ker\pi=\operatorname{span}_{\Z}\{\text{selected edge rows}\},
 \qquad E/\ker\pi\simeq F.
\]
The map is onto because it fixes every retained generator. This
description also shows directly that the old cores remain linearly
independent, even when several rooted trees attach to the same
formal root: the root contributes no extra generator or equation.

Equivalently, order eliminated coordinates by a leaf-removal order
and use the corresponding edge rows in that order. Their submatrix
on those coordinates is triangular with diagonal entries $\pm1$.
Its determinant is therefore $\pm1$. This is the integer linear
algebra behind the substitution: the selected rows remove exactly
the intended coordinates, with no finite torsion left over. By
contrast, a pivot of absolute value greater than one would not
justify this conclusion over $\Z$.

The same calculation has a useful matrix form. Order the coordinates
with the $N-q$ eliminated coordinates first and all retained
coordinates afterwards. The matrix of the selected rows is
$[T\ B]$, where $T$ is triangular with diagonal entries $\pm1$.
It follows that $T^{-1}$ is an integer matrix. Integer row operations
therefore replace these generators by $[I\ W]$, with $W=T^{-1}B$,
without changing their integer span. If $(u,v)$ is any row vector in
these coordinates, its image in the retained basis is
\[
 (u,v)\longmapsto v-uW,
 \qquad (u,v)-u[I\ W]=(0,v-uW).
\]
This identity proves both directions of the kernel claim: the image
is zero if and only if $(u,v)$ is an integer combination of the
selected rows. It also makes surjectivity explicit, since $(0,v)$
maps to $v$. The columns of $W$ that correspond to cores contain
signed sums of the edge corrections along a path; the columns for
retained active coordinates record the one surviving anchor of an
unrooted component. Thus the matrix calculation and the path
calculation describe the same quotient. For the height estimate we
use the latter description, which retains the path-length information
that a general matrix inverse estimate would discard.

As a numerical illustration of this integer substitution, suppose
two selected edge rows, on a path with retained endpoint $x_3$, are
\[
 x_1-x_2+2c_1-2c_2,
 \qquad x_2-x_3-c_1+c_2,
\]
where $c_1,c_2$ are permanent cores. This illustrates the quotient
calculation once such rows have been selected; it is not an
additional assumption on which collisions occur in $A$. Both rows
have degree zero. Substitution gives
\[
 x_2=x_3+c_1-c_2,
 \qquad x_1=x_3-c_1+c_2.
\]
In the column order $(x_1,x_2\mid x_3,c_1,c_2)$, their matrix and
its integer row reduction are
\[
 \begin{pmatrix}1&-1&0&2&-2\\0&1&-1&-1&1\end{pmatrix}
 \ \longmapsto\ 
 \begin{pmatrix}1&0&-1&1&-1\\0&1&-1&-1&1\end{pmatrix}.
\]
The first row has merely been replaced by the sum of the two rows.
For example, an old point expression $x_1+3c_1-3c_2$ becomes
$x_3+2c_1-2c_2$. Its coefficients still sum to one, although some
are negative. This is why the invariant concerns the sum of the
coefficients and their absolute values, without requiring a convex
combination of anchors. If the path also has the root row
$x_3-c_2$, the same point becomes $2c_1-c_2$, again of degree one,
and no active coordinate remains. In each case the displayed
substitutions leave the retained generators unrestricted.

For an unrooted component merge its physical blocks and use the
old actual anchor of $x_r$ as the merged anchor. Its offset is zero
because $x_r$ itself was retained. All blocks in a rooted component
become inactive. Promote every heavy merged block, keeping its
actual anchor as the new core anchor. The other merged blocks are
still pairwise disjoint and have size less than $\ell_*$. There
are at most $2m/3$ of them. Thus all structural invariants persist.

We next check the offsets. Before sample promotion every selected
unit row has core part of norm at most $8H$. A newly usable actual
collision row has norm at most $4H$ before that promotion, and at
most four active occurrences can be reclassified as cores. A
previous unit row has at most two such occurrences. Hence the
uniform bound
\[
 8H+4\le12H
\]
is valid for every edge after the sample promotions. The expression
for an eliminated coordinate follows a path of length at most three
to its retained anchor, or at most four to the root. Adding the
original point offset gives norm at most
\[
 H+4\cdot12H=49H.
\]
Heavy promotion adds at most one core unit vector, so gives at most
$49H+1\le50H$. Points in sample-promoted blocks have norm at most
$H+1\le2H$, and older inactive points have unchanged expressions.
Thus $60H$ is a convenient common bound for the next state.

There is no hidden factor equal to a component's number of vertices
in this estimate. To compute the new offset of a particular point,
only the path from its old block to the retained anchor or root is
used. Different points may use different paths, but the bound is a
maximum over point norms, not a sum over all points. This is exactly
where the short-forest lemma turns a local elimination into a
uniform height bound.

We can summarize a round by a recurrence. Let $c,m,H$ denote the
numbers of cores and active blocks and the height bound before the
round. Let $s$ be the number of sampled blocks promoted, and let
$h$ be the number of heavy components promoted after elimination.
With $q$ as above, the next state has
\[
 c'=c+s+h,\qquad m'=q-h\le\frac{2m}{3},\qquad H'\le60H,
 \qquad 0\le s\le6144\sqrt m.
\]
Promotion changes $c$ and $m$ but preserves their sum. Elimination
changes their sum by $-(N-q)$, where $N=m-s$. Indeed
$c'+m'=c+s+q=(c+m)-(N-q)$. These identities distinguish the
cost of retaining an anchor from the rank removed by adding a
relation. The heavy cost $h$ will be bounded across all rounds by
charging original points, rather than by charging the active count
separately in every round.

\medskip\noindent\emph{Step 5: counting cores and bounding the height.}
Repeat while the active count is at least $2^{32}$, then promote
every remaining block. The active count decreases by a factor of
at least $2/3$ at each nonterminal round, so the procedure terminates.
If $m_j$ is its value before round $j$, starting with $j=0$, then
$m_j\le n(2/3)^j$. Consequently sample promotions contribute fewer
than
\[
 6144\sum_{j\ge0}\sqrt{n(2/3)^j}
 =\frac{6144\sqrt n}{1-\sqrt{2/3}}<36864\sqrt n
\]
cores. The final numerical inequality follows from $\sqrt{2/3}<5/6$.

Here a sample promotion costs one basis coordinate per sampled
block, regardless of how many original points that block contains.
Later rounds are sampled only from the remaining active blocks.
In particular, permanent cores are never counted again as newly
sampled blocks. This is why a geometric sum of active counts, rather
than the total number of all coordinates at each stage, is the
appropriate quantity.

Each heavy promotion consumes at least $\ell_*$ original points.
The sets consumed by different heavy promotions, even in different
rounds, are disjoint: their points become inactive and can never
enter a later active block. Thus heavy promotions contribute at
most $n/\ell_*$ cores. Since $\log_2p/3\ge2$, the inequality
$\lfloor u\rfloor\ge u/2$ for $u\ge2$ gives
\[
 \ell_*\ge\frac{\log p}{6\log2}\ge\frac{\log p}{6}.
\]
Terminal promotion contributes fewer than $2^{32}$ cores. A NUS
set has at least two points, and hence
\[
 K\le36864\sqrt n+\frac{6n}{\log p}+2^{32}
 \le2^{33}\left(\sqrt n+\frac n{\log p}\right).
\]

Let $R$ be the number of nonterminal rounds. If $R\ge1$, then
$n\ge2^{32}$ and
\[
 R\le1+\frac{\log n}{\log(3/2)},
 \qquad \frac{\log60}{\log(3/2)}<11;
\]
the second inequality is equivalent to $(3/2)^{11}>60$.
The offset bound before terminal promotion is $60^R$. Adding the
last active unit coordinate therefore gives
\[
 \norm{z_a}_1\le60^R+1\le120n^{11}\le n^{18},
\]
where the last inequality follows from $n\ge2^{32}$.
If $R=0$, every block is still a singleton when it is promoted, so
each final point coordinate is a unit vector and the bound is immediate.

The two geometric estimates have different roles. The decreasing
sequence $m_j\le n(2/3)^j$ makes the total number of sampled anchors
a convergent sum of order $\sqrt n$. The increasing sequence
$H_j\le60^j$ is not summed: only its value at the last round is used.
Since the number of rounds is logarithmic, this last value is
polynomial in $n$. In formulas, when $R\ge1$ the passage from rounds
to height is
\[
 60^R\le60\exp\!\left(
       \frac{\log60}{\log(3/2)}\log n\right)\le60n^{11}.
\]
Using a sum of point heights, or allowing paths of length $m_j$ in
place of the fixed bound four, would give a different recurrence.
The uniform pointwise invariant is what permits the present one.

Finally, each surviving basis vector is the image of an actual point;
two such anchors cannot be the same point because basis vectors are
distinct. The map $\deg(e_a)=1$ annihilates every collision row and
therefore descends through $P\subseteq\Lambda$. It takes value one
on every anchor basis vector, so in these coordinates it is the
coordinate-sum map. Since $\deg(e_a)=1$, we obtain
$\sum_i(z_a)_i=1$ for every $a\in A$, as required.
\end{proof}

\section{Proof of the lower bound}\label{sec:lower}
\begin{proof}[Proof of Theorem~\ref{thm:main-lower}]
Write $X=\log p$ and first suppose that $A$ is inclusion-minimal,
with $n=|A|$. Apply Theorem~\ref{thm:compression}, obtaining
$P\subseteq\Lambda$ and a free quotient $\Z^K$ with actual anchors
$a_1,\ldots,a_K$. We shall compare the rational and modular ranks
of the collision images in this smaller quotient.

Let $q:\Z^A\to\Z^K$ denote the quotient map. Its restriction
to $\Lambda$ has kernel $P$, because $P\subseteq\Lambda$ and
$\ker q=P$. Hence the relevant image is exactly
$q(\Lambda)\simeq\Lambda/P$. Upon taking rational spans this
gives the exact sequence
\[
 0\longrightarrow P\otimes\Q
 \longrightarrow\Lambda\otimes\Q
 \longrightarrow q(\Lambda)\otimes\Q\longrightarrow0.
\]
The rank subtraction below is the dimension formula for this
sequence. Modular singularity will be established separately by
an explicit vector of anchor labels, so no exactness assertion
after reduction modulo $p$ is required.

One can keep track of these maps by a single integer matrix. Let $M$
again denote the original collision matrix, and let $Z$ be the
$n\times K$ matrix whose row indexed by $a$ is $z_a$. Then
\begin{equation}\label{eq:coordinate-identities}
 \begin{gathered}
 Z_{a_i,*}=e_i,\qquad Z\one_K=\one_n,\\
 Z(a_1,\ldots,a_K)^{\mathsf T}=(a)_{a\in A}^{\mathsf T}
       \quad\hbox{over }\F_p.
 \end{gathered}
\end{equation}
The first identity says that the anchor rows form an identity matrix;
the second records degree; the third records the actual labels.
Moreover, for every original point there is the integer relation
\[
 e_a-\sum_{i=1}^K(z_a)_i e_{a_i}\in P.
\]
It follows by applying the quotient map to the left-hand side, and
uses its exact kernel $P$. Thus these are identities in the constructed
integer quotient, rather than identities obtained only after passing
to rational spans. The compressed collision matrix is $MZ$. Its rows
are the images of the original rows, and \eqref{eq:coordinate-identities}
gives both $MZ\one_K=0$ and
$MZ(a_1,\ldots,a_K)^{\mathsf T}=0$ in their respective coefficient
fields. The rank calculation below supplies the additional fact
that its rational rank is exactly $K-1$.

\medskip\noindent\emph{A nonzero determinant divisible by $p$.}
Since $\Z^A/P$ has rank $K$, the subgroup $P$ has rank $n-K$.
Lemma~\ref{lem:complete-rank} gives $\rk_\Q\Lambda=n-1$, whence
\[
 \rk_\Q(\Lambda/P)=(n-1)-(n-K)=K-1.
\]
The actual collision images span this rational space because their
preimages generate $\Lambda$. Every such image has coordinate sum
zero. Deleting the last coordinate is injective on the subspace
with coordinate sum zero: the missing coordinate is the negative
sum of the others. We may therefore choose $K-1$ actual collision
images whose first $K-1$ coordinates form a square integer matrix
$D$ with $\det D\ne0$.

The two operations in this argument should be distinguished. We first
pass to the quotient by the specifically constructed subgroup $P$,
which removes $n-K$ independent rational directions from $\Lambda$.
We then delete one column of its image, which removes no row-space
dimension because every row has coordinate sum zero. Thus the size
$K-1$ of $D$ comes from an exact rank calculation. Choosing any
convenient collection of short collision rows without this rank
calculation would not establish that its determinant is nonzero.

The actual-label homomorphism $e_a\mapsto a\in\F_p$ annihilates
$\Lambda$, so it descends to the quotient and takes value $a_i$
on its $i$th basis vector. In particular $K\ge2$: if $K=1$, the
degree-one condition would give $z_a=(1)$ for every $a$, forcing
all labels to equal the single anchor label. This contradicts
$n\ge2$ and the distinctness of the points.

For a collision image $r=(r_1,\ldots,r_K)$, we have
$\sum_i r_i=0$ over $\Z$ and $\sum_i r_i a_i=0$ over $\F_p$.
Consequently
\[
 \sum_{i<K}r_i(a_i-a_K)=0\quad\hbox{in }\F_p.
\]
Thus $D$ annihilates the nonzero vector
$(a_1-a_K,\ldots,a_{K-1}-a_K)$ modulo $p$. This vector is nonzero
because the anchors are distinct. It follows that $p\mid\det D$,
and the nonzero integer determinant satisfies $p\le|\det D|$.

For example, when $K=3$, two independent compressed rows have the
form $(u,v,-u-v)$ and $(w,z,-w-z)$, so deletion produces
\[
 D=\begin{pmatrix}u&v\\w&z\end{pmatrix},\qquad
 \det D=uz-vw.
\]
Rational independence says that this integer is nonzero. The
label identities say
\[
 D\begin{pmatrix}a_1-a_3\\a_2-a_3\end{pmatrix}=0
       \quad\hbox{over }\F_p.
\]
Multiplying by the adjugate matrix shows that $\det D$ times each
coordinate of this vector is zero modulo $p$. At least one
coordinate is nonzero, so the field property forces
$\det D=0$ modulo $p$. This small case exhibits both ingredients
of prime detection: rational independence excludes the zero
integer determinant, and the nonconstant anchor labels force its
divisibility. The preceding argument applies identically to any
$K\ge2$.

\medskip\noindent\emph{The determinant height estimate.}
Each collision image is a signed sum of four point coordinates,
so its $\ell_1$-norm is at most $4n^{18}$. Deleting a coordinate
does not increase that norm. For any square matrix, the determinant
expansion and the triangle inequality give a bound by the product
of the row $\ell_1$-norms: the permutation products are a subset
of the products obtained by independently choosing one entry in
each row. There is therefore no extra factorial factor. We obtain
\[
 X\le\log|\det D|
 \le(K-1)\log(4n^{18})\le20K\log n,
\]
where $\log4\le2\log n$ because $n\ge2$.

To spell out the determinant estimate, put $d=K-1$ and let
$D_{ij}$ be the entries of $D$. Then
\[
 |\det D|
 \le\sum_{\sigma\in S_d}\prod_{i=1}^d|D_{i,\sigma(i)}|
 \le\sum_{j_1,\ldots,j_d=1}^d\prod_{i=1}^d|D_{i,j_i}|
 =\prod_{i=1}^d\sum_{j=1}^d|D_{ij}|.
\]
The middle sum allows repeated column choices, and all its terms
are nonnegative, which justifies the second inequality. Each row
sum is at most $4n^{18}$. This bound needs only the maximum point
height produced by compression; it does not require the chosen
collision rows to have disjoint supports.

There is also a distinction between the rows used to define $P$ and
the rows selected for $D$. Some defining rows of $P$ are differences
of collisions, and repeated substitutions could make a description
of those rows in an intermediate presentation cumbersome. The rows
of $D$, however, are chosen afresh from the images of original
collisions. They therefore still have the form
$z_a+z_b-z_c-z_d$ and obey the four-point height bound above.
This selection is possible because the images of all original
collisions span $\Lambda/P$. We need no bound on the coefficients
that express an arbitrary element of $\Lambda/P$ in a lattice basis.

In particular, rational independence is used only to select rows.
We do not rescale those rows, clear denominators in them, or replace
them by rational linear combinations when forming $D$. The selected
entries remain integers of the same bounded height. The identity
$p\mid\det D$ then concerns that actual integer determinant. This
separates the use of rational linear algebra to find a nonzero minor
from the integer size estimate applied to the minor once selected.

Using the dimension bound and the relaxed constant $C=2^{38}$,
for which $20\cdot2^{33}\le C$, gives
\begin{equation}\label{eq:lower-implicit-detailed}
 X\le C\left(\sqrt n+\frac nX\right)\log n.
\end{equation}

\medskip\noindent\emph{Solving the inequality.}
In the stated range $p\ge64$, we have $X>4$ and $\log X>1$.
If $n\ge X^2$, then certainly
$n\ge2^{-80}(X/\log X)^2$. In the remaining case $n<X^2$,
we have $n/X<\sqrt n$ and $\log n<2\log X$. Substitution into
\eqref{eq:lower-implicit-detailed} yields
\[
 X\le4C\sqrt n\log X.
\]
All factors are positive, so division and squaring are legitimate:
\[
 n\ge\frac1{16C^2}\left(\frac X{\log X}\right)^2
 =2^{-80}\left(\frac{\log p}{\log\log p}\right)^2.
\]
Finally, an arbitrary finite NUS set contains an inclusion-minimal
NUS subset. Applying the established bound to that subset and using
monotonicity of cardinality proves the same lower bound for $A$.
\end{proof}

The remaining logarithmic loss occurs at a specific step: the quotient
has $O(\sqrt n+n/\log p)$ coordinates, but their polynomial height
costs $O(\log n)$ per row in the determinant estimate. Unit pivots
ensure that the compression is valid over the integers; the present
height estimate does not eliminate this cost.

More explicitly, in the range $n<X^2$ the dimension estimate is
$K=O(\sqrt n)$, whereas the determinant estimate is
$X\le O(K\log n)$. Replacing $\log n$ by $2\log X$ therefore
gives a lower bound for $\sqrt n$ of order $X/\log X$, and squaring
produces the two powers of $\log\log p$ in the denominator.
The exponent limit does not remove these factors: their contribution
to $\log m(p)/\log\log p$ tends to zero even though the factors
themselves are unbounded.

\begin{question}\label{qu:quadratic}
Is there an absolute constant $c>0$ such that
\[
 m(p)\ge c(\log p)^2
\]
for every sufficiently large prime $p$?
\end{question}
An affirmative answer, together with Theorem~\ref{thm:main-upper},
would give $m(p)=\Theta((\log p)^2)$. Corollary~\ref{cor:exponent}
determines the logarithmic exponent but leaves this question open.

\appendix
\section{A sparse ternary upper construction}\label{app:upper}
This appendix is independent of the compression argument. The
balanced-sumset transfer is recorded by Bedert~\cite{Bedert}; the
weakly ternary-balanced version appears in Cao--Yuan~\cite[Appendix~A]{CY}.
We give the transfer proofs for completeness, then construct a sparse
correction set that yields the second coefficient in
Theorem~\ref{thm:main-upper}.

\subsection{From balanced seeds to sets with no unique sums}

\begin{definition}A nonempty set $B\subseteq\F_p$ is balanced if every $b\in B$ admits
\[
 2b=u+v,
 \qquad u,v\in B\setminus\{b\},\quad u\ne v.
\]
\end{definition}

Nedev~\cite{Nedev} gives an algorithm constructing balanced subsets
of $\F_p$ of size at most $(1+o(1))\log_2p$. Together with the
following transfer, this yields an upper bound of order $(\log p)^2$
for the NUS problem. The ternary construction below improves the
leading constant obtained in this way.

\begin{lemma}\label{lem:balanced-sum}
If $B\subseteq\F_p$ is nonempty and balanced, then $B+B$ is NUS.
\end{lemma}
\begin{proof}
Take any unordered pair of elements of $B+B$, allowing equal entries,
and choose seed representations so that it is $\{a+b,c+d\}$ with
$a,b,c,d\in B$. Recombining the four seed entries gives two pairs
with the same total:
\[
 \{a+c,b+d\},\qquad \{a+d,b+c\}.
\]
If either is different from the original unordered pair, it is
already the required repair. We examine the case where both fail.

Equality of the first pair with the original pair can match its
entries in either order. Cancellation gives $b=c$ or $a=d$,
respectively. Equality of the second pair gives $b=d$ or $a=c$.
The four possible combinations are
\[
 \begin{array}{c|c}
 \text{equalities forced by the two matchings}
   &\text{three coincident entries}\\ \hline
 b=c,\ b=d&b=c=d\\
 b=c,\ a=c&a=b=c\\
 a=d,\ b=d&a=b=d\\
 a=d,\ a=c&a=c=d.
 \end{array}
\]
This case distinction is valid for unordered multisets, including
when one or both displayed pairs are diagonal.
Thus at least three of $a,b,c,d$ coincide. Denote their common
value by $x$ and the remaining value by $y$; if all four coincide,
take $y=x$. The original pair is then $\{2x,x+y\}$.

Balance at $x$ supplies $u+v=2x$, with $u,v$ distinct and neither
equal to $x$. At most one of these two distinct values can equal
$2x-y$. We may thus choose one of them, denoted by $z$, outside
$\{x,2x-y\}$, and let $w$ denote the other. Both $z+y$ and $w+x$
belong to $B+B$, and
\[
 (z+y)+(w+x)=(z+w)+x+y=3x+y.
\]
This is the total of $\{2x,x+y\}$. Moreover $z+y\ne2x$ because
$z\ne2x-y$, and $z+y\ne x+y$ because $z\ne x$. Thus the new
unordered pair cannot equal the old one. The reasoning also covers
$x=y$, so diagonal pairs have been repaired as well. Since the
original pair was arbitrary, $B+B$ is NUS.
To check the cardinality requirement explicitly, fix any $b_0\in B$.
Its midpoint certificate shows that $B$ has at least three elements,
and the injective translation $b\mapsto b+b_0$ maps $B$ into
$B+B$. Thus the sumset has at least two elements as required in
the definition of NUS.
\end{proof}

\subsection{Weak ternary balance}
\begin{definition}
A nonempty $C\subseteq\F_p$ is weakly ternary-balanced if every $x\in C$ has a certificate
\[
 3x=u+v+w,\qquad u,v,w\in C,
 \qquad\{u,v,w\}\ne\{x,x,x\}.
\]
\end{definition}
\begin{lemma}\label{lem:ternary-transfer}
For $p>3$, a weakly ternary-balanced $C$ has $C+C$ NUS.
\end{lemma}
\begin{proof}
For any pair of $C+C$, the two recombinations used in the proof
of Lemma~\ref{lem:balanced-sum} either give a repair immediately
or reduce the problem to $\{2x,x+y\}$ with $x,y\in C$.
Choose a nontrivial ternary certificate $3x=u+v+w$ in $C$.
We first show that one of its entries lies outside
$\{x,2x-y\}$.

If $x=y$, this forbidden set is just $\{x\}$, and the assertion
is exactly the nontriviality of the certificate. Otherwise set
$y'=2x-y\ne x$ and suppose all three entries belong to $\{x,y'\}$.
Let $j$ be the number of entries equal to $y'$. The certificate
would give
\[
 3x=(3-j)x+j(2x-y)=3x+j(x-y).
\]
The case $j=0$ is trivial and excluded. For $1\le j\le3$, the
assumption $p>3$ makes $j$ nonzero in $\F_p$, so this equality
would force $x=y$, a contradiction. This proves the assertion
and explains precisely where the characteristic restriction is used.

Choose such an entry $z$. The other two certificate entries have
sum $3x-z$, so this element lies in $C+C$. The pair
\[
 \{z+y,3x-z\}
\]
also lies in $C+C$ and has total $3x+y$. Its first entry differs
from $2x$ and from $x+y$, by the choice of $z$. Hence it differs
as an unordered pair from $\{2x,x+y\}$ and is the desired repair.
Finally, a nonempty weakly ternary-balanced seed cannot be a singleton:
all three entries of a certificate would then equal its only point.
Translation by any fixed seed point injects $C$ into $C+C$, so
$|C+C|\ge|C|\ge2$. Together with the repairs above, this verifies
all parts of the NUS definition.
\end{proof}

\subsection{The finite construction}
\begin{theorem}[Finite sparse-core construction]\label{thm:sparse-core}
Let $p$ be prime, let $t\ge4$ be divisible by four, put $b=3^{t/2}$ and $M_0=3^t=b^2$, and assume $p>30M_0$. There is a weakly ternary-balanced $C\subseteq\F_p$ with
\begin{equation}\label{eq:sparse-size}
 |C|\le17b+4t+20+
 2(t+1)\left(\left\lfloor\log_{3M_0^2}p\right\rfloor+1\right).
\end{equation}
Consequently $m(p)\le\binom{|C|+1}{2}$.
\end{theorem}
\begin{proof}
We construct integer representatives in $[0,p)$ and then reduce them
modulo $p$. Throughout the construction, every point ever inserted
is retained. A moving interval records where subsequent points will
be chosen; shrinking this interval does not discard earlier points.
This distinction ensures that certificates involving retired
endpoints remain available in the final set.

Here is the design of the construction. The core has only $5b$
points but its pair sums realize a whole interval of corrections
of length comparable with $b^2=M$. Those corrections impose the
congruences needed for $t$ successive divisions by three. Two
alternating transitions use at most $2(t+1)$ new points and reduce
the width by a factor approximately $3M^2$. The final completion
handles the only two core points lacking certificates. We first
verify this finite procedure, including a possibly incomplete last
pair of transitions, and only afterwards choose $t$ as a function
of $p$.

For readability write $M=M_0=3^t=b^2$. Since $t\ge4$ is divisible
by four, $b\ge9$ is an odd integer divisible by three and $M\ge81$.
Write $p=3a+\rho$ with $\rho\in\{1,2\}$, and put
\[
 B_0=b+1,\qquad Q_0=(3b-1)/2,
\]
\[
 K_0=\{a+jB_0:-Q_0\le j\le Q_0\}
       \cup([a-B_0,a+B_0]\cap\Z).
\]

\medskip\noindent\emph{Step 1: the core and the initial endpoints.}
The progression has $2Q_0+1=3b$ points and the short interval has
$2B_0+1=2b+3$ points. Their intersection consists of exactly
$a-B_0,a,a+B_0$, so $|K_0|=5b$. Its extreme points are
\[
 u=a-Q_0B_0,\qquad v=a+Q_0B_0,
 \qquad Q_0B_0=\frac{3b^2+2b-1}{2}\le2M.
\]
Every nonextreme progression point is the midpoint of its two
neighbours at distance $B_0$. Every interior point of the short
interval is the midpoint of its neighbours at distance one. The
two ends of that interval are also nonextreme progression points,
since $Q_0\ge2$. Thus every point of $K_0$ except $u,v$ has a
nontrivial midpoint certificate within $K_0$. Adding the point
itself as a third entry converts $2x=r+s$ into $3x=r+s+x$.
Only the two extreme core points will need a later completion.

Choose $s\in\{-1,0,1\}$ so that $3\mid p+a+s$, and insert
\[
 L_0=0,\qquad R_0=(p+a+s)/3.
\]
The short interval in the core contains $a+\rho$ and $a+s$.
The required modular certificates are
\[
 3L_0\equiv a+a+(a+\rho),\qquad
 3R_0\equiv0+0+(a+s)\pmod p.
\]
Indeed the first sum equals $p$, and the second congruence follows
from $3R_0=p+a+s$. These certificates are nontrivial. To check the
positions, $p>30M$ gives $a>(30M-2)/3$, hence $a>2M$.
Also
\[
 R_0-a=\frac{a+\rho+s}{3}>2M,
 \qquad R_0<p.
\]
Together with the core radius bound this proves
$0=L_0<u<v<R_0<p$.

\medskip\noindent\emph{Step 2: two core points realize every correction.}
For any integer $\delta$ with $|\delta|\le(3M-1)/2$, use Euclidean
division by the positive integer $B_0$:
\[
 j=\left\lfloor\frac{\delta}{B_0}\right\rfloor,
 \qquad r_0=\delta-jB_0,\qquad 0\le r_0<B_0.
\]
Since $(3M-1)/2\le Q_0B_0$, we have $-Q_0\le j\le Q_0$.
Consequently both
\[
 c=a+jB_0,\qquad d=a+r_0
\]
are actual core points, and $c+d=2a+\delta$. Their rational mean
\[
 \bar d=\frac{c+d}{2}=a+\frac\delta2
\]
satisfies $|\bar d-a|\le3M/4$. The number $\bar d$ is used only
to estimate positions. It need not be integral and is never inserted
as a new point. The correction itself always uses the two actual
points $c,d$.

This use of Euclidean division also handles negative corrections
without a separate construction. For example, at the smallest
allowed scale $t=4$, one has $b=9$, $M=81$, $B_0=10$ and
$Q_0=13$. The extreme correction $\delta=-121$ is
$(-13)\cdot10+9$, and is realized by the core points $a-130$
and $a+9$. Thus the progression supplies the coarse multiple
and the short interval supplies the remainder. Their individual
sizes are $O(b)$, while the range covered by their sums is $O(b^2)$.

\medskip\noindent\emph{Step 3: the interval states and their transitions.}
Set $\alpha=3/4$, $\beta=9/10$ and
\[
 z_\theta(L,R)=(1-\theta)L+\theta R.
\]

The choices of $\alpha$ and $\beta$ have two purposes, which can
be separated algebraically. First consider an ideal transition in
which the core mean $\bar d$ equals the current centre $a$, with
no positioning error. In an interval with $a=L+\alpha D$, the
right correction $q=(R+2a)/3$ gives
\[
 q-L=\frac{1+2\alpha}{3}D,\qquad
 \frac{a-L}{q-L}=\frac{3\alpha}{1+2\alpha}=\beta.
\]
Starting instead with $a=L+\beta D$, the left correction
$\ell=(L+2a)/3$ gives
\[
 R-\ell=\frac{3-2\beta}{3}D,\qquad
 \frac{a-\ell}{R-\ell}=\frac{\beta}{3-2\beta}=\alpha.
\]
For $\alpha=3/4$ these formulas give $\beta=9/10$, with raw
width factors $5/6$ and $2/5$. Second, both resulting relative
positions have periodic ternary expansions using only the digits
$0$ and $2$. This permits repeated subdivision using an endpoint
already present and one newly inserted endpoint at each step.
The exact estimates below allow $\bar d$ and the centre to differ
from $a$; the ideal calculation is only an explanation of the
parameter choice, not an extra assumption on the construction.

The current interval has integer endpoints and carries one of these
two states. We maintain
\begin{equation}\label{eq:core-centre}
 K_0\subset(L,R),\qquad |z_\theta(L,R)-a|\le3M/4.
\end{equation}
Initially the state is $\alpha$, and a direct calculation gives
\[
 z_\alpha(L_0,R_0)-a
 =\frac{p+a+s}{4}-a=\frac{\rho+s}{4},
\]
so the invariant holds. Before executing any transition, put
$D=R-L$ and stop if $D<60M^2$.

In an $\alpha$ state, select the unique symmetric representative
\[
 \delta\equiv3L-R-2a\pmod{3M},
 \qquad |\delta|\le(3M-1)/2.
\]
Such a representative exists because $3M$ is odd. Obtain $c,d$
from Step 2 and insert
\[
 q=(R+c+d)/3=(R+2a+\delta)/3.
\]
The congruence implies that $q$ is an integer and $M\mid q-L$.
It also gives the certificate $3q=R+c+d$.
More precisely, for some integer $h$ the chosen congruence says
$\delta=3L-R-2a+3Mh$, and hence
\[
 R+2a+\delta=3L+3Mh,\qquad q=L+Mh.
\]
Thus the single congruence modulo $3M$ enforces both the first
division by three and all $t$ subsequent ternary subdivisions.
Choosing a correction only modulo $M$ would guarantee merely
$M/3\mid q-L$, leaving one fewer guaranteed subdivision.
The relevant position and centre identities are
\begin{align*}
 q-L&=\frac56D+\frac23\bigl(\bar d-z_\alpha(L,R)\bigr),\\
 z_\beta(L,q)&=\frac25z_\alpha(L,R)+\frac35\bar d.
\end{align*}
To verify the first identity, use $R=L+D$ and
$z_\alpha(L,R)=L+3D/4$ in
\[
 q-L=\frac{D+2(\bar d-L)}3
     =\frac{D+2(\bar d-z_\alpha(L,R))+3D/2}3.
\]
For the second, substitute $q=(R+2\bar d)/3$ into
$z_\beta(L,q)=L/10+9q/10$. The coefficients of $L,R,\bar d$
become $1/10,3/10,3/5$, respectively, as on the right-hand side.
The difference in parentheses has absolute value at most $3M/2$.
Thus the raw width differs from $5D/6$ by at most $M$. Since
$D\ge60M^2$, it is strictly between zero and $D$, proving
$L<q<R$. In particular the certificate of $q$ is nontrivial,
because it contains the distinct point $R$. The centre identity
is a convex combination, so it preserves the error bound $3M/4$.

Next subdivide $[L,q]$ into thirds $t$ times, retaining the left
third for a digit $0$ and the right third for a digit $2$ in the
periodic word $2200$. If the current endpoints are $L',R'$, a
left-third step inserts $(2L'+R')/3$ and a right-third step inserts
$(L'+2R')/3$. Their certificates are, respectively,
$(L',L',R')$ and $(L',R',R')$. Each is nontrivial since the new
point is strictly between the distinct old endpoints. All old
endpoints are retained in the set.

These subdivision points are integers: the raw width is divisible
by $3^t=M$, and each step divides the width by three, leaving the
divisibility required for the next step. More explicitly, write the
raw width as $W_0=3^t h$ with $h$ a positive integer. Before the
$j$th subdivision, for $1\le j\le t$, the width is
$3^{t-j+1}h$, which is divisible by three. The inserted endpoint
is the current left endpoint plus either one third or two thirds
of this width. It is therefore an integer strictly between the
current endpoints. After the last subdivision the width is $h>0$,
so the interval has not collapsed.
The real number
$\beta=.22002200\ldots{}_3$ lies at relative position $\beta$
again after any multiple of four such digits. Since $4\mid t$,
the final interval has the same $\beta$ centre as $[L,q]$ and
width $(q-L)/M$. This is an identity about affine positions, not
an additional point inserted into the set. For example, four
digits send $[0,81]$ to $[72,73]$, and
$72+(9/10)(73-72)=(9/10)81$.

In a $\beta$ state choose instead
\[
 \delta\equiv3R-L-2a\pmod{3M},
 \qquad |\delta|\le(3M-1)/2,
 \qquad \ell=(L+c+d)/3.
\]
Then $\ell$ is integral, $M\mid R-\ell$, and $(L,c,d)$ is its
certificate. The corresponding identities are
\begin{align*}
 R-\ell&=\frac25D+\frac23\bigl(z_\beta(L,R)-\bar d\bigr),\\
 z_\alpha(\ell,R)&=\frac56z_\beta(L,R)+\frac16\bar d.
\end{align*}
The integrality check is the reflected version of the previous
one: writing $\delta=3R-L-2a+3Mh$ gives $\ell=R+Mh$, so
$R-\ell=-Mh$ is divisible by $M$. The position estimate below
ensures that this width is positive; no sign for $h$ was assumed
when selecting the symmetric representative.
For the displayed identities, the first follows from
\[
 R-\ell=\frac{2(R-\bar d)+D}3
        =\frac{2D}{5}+\frac23(z_\beta(L,R)-\bar d),
\]
because $R-z_\beta(L,R)=D/10$. Also
$z_\alpha(\ell,R)=\ell/4+3R/4=L/12+\bar d/6+3R/4$,
which is the claimed convex combination. The two different
contraction factors, $5/6$ and $2/5$, have product $1/3$;
this is the extra factor of three gained per complete slot.
Again the raw width error is at most $M$, so $L<\ell<R$ and the
certificate is nontrivial. Subdivide $[\ell,R]$ using $t$ digits
of the periodic word $20$. The identity
$\alpha=.202020\ldots{}_3$ and the evenness of $t$ ensure that
the final $\alpha$ centre is unchanged. For four digits the model
interval $[0,81]$ becomes $[60,61]$, with
$60+(3/4)(61-60)=(3/4)81$. Each transition therefore inserts at
most one correction point and $t$ subdivision endpoints.

The periodic-coordinate assertion can be checked without any
infinite-expansion convention. If an interval of width $W$ is
replaced by its left or right third according to a digit
$d\in\{0,2\}$, a point with old relative coordinate $x$ has new
relative coordinate $3x-d$. For the word $2200$, applying these
four transformations to $\beta=9/10$ returns $\beta$; equivalently,
\[
 \beta=\frac2{3}+\frac2{3^2}+\frac{\beta}{3^4}.
\]
For the word $20$, the corresponding identity is
$\alpha=2/3+\alpha/3^2$. These equations prove the centre
preservation used above. Divisibility of $t$ by four accommodates
both words, while $M=3^t$ supplies the required integer endpoints.

More explicitly, for a starting interval $[L,L+W]$ and digits
$d_1,\ldots,d_t\in\{0,2\}$, successive choices of thirds give
\[
 L_t=L+W\sum_{j=1}^t\frac{d_j}{3^j},\qquad
 R_t-L_t=\frac{W}{3^t}.
\]
These formulas follow by induction: the $j$th choice moves the
left endpoint by $d_j W/3^j$ and divides the current width by
three. If $\xi$ is the appropriate periodic coordinate and the
number of digits is a multiple of its period, then
\[
 \xi=\sum_{j=1}^t\frac{d_j}{3^j}+\frac{\xi}{3^t},
 \qquad L_t+\xi(R_t-L_t)=L+\xi W.
\]
Thus the centre is preserved exactly, while the width is divided
by $M$. In particular, there is no accumulated centre error from
the $t$ subdivisions; the positioning error arises only in the
single correction step preceding them.

\medskip\noindent\emph{Step 4: preservation and termination.}
It is useful to separate the error in the centre from the error in
the width. Write $e=z_\theta(L,R)-a$ for the current centre error
and $f=\bar d-a$ for the chosen core-mean error. Both have absolute
value at most $3M/4$. The identities in Step 3, together with exact
centre preservation under subdivision, give the following recurrences:
\[
 \begin{array}{c|c|c}
 \text{transition}&\text{new centre error }e'
      &\text{width error }\epsilon\\ \hline
 \alpha\ \text{to}\ \beta&\frac25e+\frac35f
      &\frac{2(f-e)}{3M}\\[2pt]
 \beta\ \text{to}\ \alpha&\frac56e+\frac16f
      &\frac{2(e-f)}{3M}.
 \end{array}
\]
In each row the coefficients defining $e'$ are nonnegative and sum
to one, so $|e'|\le3M/4$. Also
$|\epsilon|\le(2/(3M))(3M/2)=1$. These estimates hold at every
transition with the same constants. For example, across a complete
pair of transitions with core-mean errors $f_1,f_2$, the centre
error obeys
\[
 e_2=\frac13e_0+\frac12f_1+\frac16f_2.
\]
This is again a convex combination. Consequently one does not add
an error of size $M$ at each of the possibly many transitions:
the centre stays in a fixed interval about $a$ throughout the
construction. The division by $M$ in the width formula is equally
important, converting a raw width error of size at most $M$ into
the absolute error one used in the slot count.

After the subdivisions the two width formulas are
\begin{equation}\label{eq:width-detailed}
 D'=\frac{5D}{6M}+\epsilon
 \quad\hbox{or}\quad
 D'=\frac{2D}{5M}+\epsilon,
 \qquad |\epsilon|\le1,
\end{equation}
for an $\alpha$ to $\beta$ or a $\beta$ to $\alpha$ transition,
respectively. The centre error remains at most $3M/4$ in each
case. A core point lies within $2M$ of $a$, hence within
$11M/4$ of the new centre. In the first case
$D'\ge50M-1$; the nearer endpoint of a $\beta$ interval is
$D'/10$ from its centre, and
\[
 D'/10\ge5M-1/10>11M/4.
\]
In the second case $D'\ge24M-1$, and the nearer endpoint of
an $\alpha$ interval is $D'/4$ from its centre, with
\[
 D'/4\ge6M-1/4>11M/4.
\]
Thus all core points lie strictly inside the new interval,
proving \eqref{eq:core-centre} at the end of every transition.
The subdivision intervals between these final states need no
separate core invariant; their endpoints have their own creation
certificates and lie in the preceding raw interval.

Every new point so far lies inside the previous current interval,
and hence in $[0,R_0]$. The new width is a positive integer and
strictly smaller than $D$, as is also immediate from
\eqref{eq:width-detailed} with $M\ge81$ and $D\ge60M^2$.
There can therefore be only finitely many transitions. On stopping,
$D<60M^2$ and the core is still strictly inside the current interval.

\medskip\noindent\emph{A numerical transition.}
To illustrate the preceding formulas, take $t=4$ and the prime
$p=1000003$. Then $b=9$, $M=81$, $a=333334$, $\rho=1$,
$B_0=10$ and $Q_0=13$. We can choose $s=1$, giving the initial
interval $[0,444446]$. Its width exceeds the transition threshold
$60M^2=393660$, so the first right transition is executed.
The correction is determined modulo $3M=243$:
\[
 3L-R-2a=-1111114\equiv-118\pmod{243}.
\]
The representative $\delta=-118$ is in $[-121,121]$ and has
Euclidean decomposition $-118=(-12)\cdot10+2$. Thus the core
points $c=333214$, $d=333336$ give
\[
 3q=444446+333214+333336=1110996,
 \qquad q=370332=81\cdot4572.
\]
The four digits $2200$ act on the raw interval as follows:
\[
 \begin{array}{c|r|r|r}
 \text{digits used}&\text{left endpoint}&\text{right endpoint}
     &\text{width}\\ \hline
 \text{none}&0&370332&370332\\
 2&246888&370332&123444\\
 22&329184&370332&41148\\
 220&329184&342900&13716\\
 2200&329184&333756&4572.
 \end{array}
\]
For instance the first inserted endpoint satisfies
$3\cdot246888=0+2\cdot370332$, and the third satisfies
$3\cdot342900=2\cdot329184+370332$; these are exactly the
right-third and left-third certificates. The final interval has
\[
 z_\beta=329184+\frac9{10}\,4572
       =a-\frac{176}{5},
 \qquad \frac{176}{5}<\frac{3M}{4}.
\]
It contains the entire core, whose extremes are $333204$ and
$333464$. Its width is below $60M^2$, so the procedure stops
before the following left transition. The completion in Step~6
still has to be performed. This example shows explicitly why the
count in Step~5 must allow an incomplete final slot, while the
certificates and centre invariant already hold after the single
executed transition.

\medskip\noindent\emph{Step 5: counting complete and incomplete slots.}
Group a right transition and the following left transition into
one slot. A slot always starts with a right transition; the last
slot may stop before executing its left transition. For a completed
pair, composing \eqref{eq:width-detailed} gives
\[
 D_2\le\frac{D_0}{3M^2}+1+\frac2{5M}.
\]
Subtracting a fixed amount absorbs the additive error. More
precisely, with $A_0=3M^2$, the inequality
$1+2/(5M)\le4-4/A_0$ implies
\[
 D_2-4\le\frac{D_0-4}{A_0}.
\]
This inequality holds already for $M\ge1$.

The subtraction of four is a device for iterating a contraction
with additive errors. In general an estimate $D_2\le D_0/A+E$
implies $D_2-c\le(D_0-c)/A$ whenever $E\le c(1-1/A)$.
Here $A=3M^2$, $E=1+2/(5M)$ and $c=4$ satisfy that condition.
Moreover the widths at the start of every executed slot exceed
four, so the shifted widths used in the counting argument are
positive. No estimate for an unexecuted transition is being used.

Let $D_{\mathrm{init}}=R_0-L_0$. If slot $j$ starts after $j$
completed pairs, then iterating the shifted inequality gives
\[
 A_0^j(D_j-4)\le D_{\mathrm{init}}-4.
\]
The decision to start that slot means $D_j\ge60M^2$, so
\[
 A_0^j\le\frac{D_{\mathrm{init}}-4}{60M^2-4}<p.
\]
Thus the number of slots is at most
$\lfloor\log_{3M^2}p\rfloor+1$. Charging at most $2(t+1)$
insertions to each slot proves the last term in
\eqref{eq:sparse-size}. This charge is an upper bound even for
an incomplete final slot. The contraction inequality itself has
only been applied to pairs that were actually executed.

There are three possible endings to the procedure. If the initial
width is already below $60M^2$, no slot is executed. Otherwise, each
started slot either contains both transitions or stops after its
right transition because the new width is below the threshold.
In the latter case there can be no later slot. If $J$ slots are
started, their indices are exactly $0,\ldots,J-1$, and every such
index satisfies $j<\log_{3M^2}p$ by the displayed estimate. This
implies $J\le\lfloor\log_{3M^2}p\rfloor+1$, including the case
$J=0$. The insertion charge $2(t+1)J$ is consequently valid without
assuming that the last slot is complete.

\medskip\noindent\emph{Step 6: repairing the two extreme core points.}
The remaining issue is to give $u,v$ nontrivial certificates.
At termination start on the left with the current endpoint $y=L<u$,
whose certificate has already been recorded. While the positive
integer gap $g=u-y$ is at least $5B_0$, choose $j\in\{0,1,2\}$
so that $3\mid y+2(u+jB_0)$. This is possible because
$B_0\equiv1\pmod3$, so those three values represent all residues
modulo three. Set
\[
 c=u+jB_0\in K_0,\qquad z=(y+2c)/3.
\]
Then $z$ is an integer with certificate $(y,c,c)$, and
\[
 u-z=\frac{g-2jB_0}{3},\qquad
 0<u-z\le\frac g3.
\]
The strict positivity follows from $g\ge5B_0$ and $j\le2$.
Also $z>y$, since $c>y$. Hence $y<z<u$, the certificate is
nontrivial, and replacing $y$ by $z$ reduces the gap by at least
a factor of three. These points stay inside the terminal interval.

When $1\le u-y<5B_0$, insert every integer in $[y,u+B_0]$.
There are at most $6B_0$ such integers, so this is also an upper
bound for the number of new points. The left endpoint $y$ retains
its earlier certificate. The right endpoint $u+B_0$ has the
progression midpoint certificate with neighbours $u$ and $u+2B_0$.
Every other point of the filled interval has its two step-one
neighbours in that interval. In particular $u$ now has a nontrivial
midpoint certificate. Converting each midpoint certificate to a
ternary certificate supplies all required witnesses on this side.

For the right side, start with $y=R>v$, choose
$c=v-jB_0$ with $j\in\{0,1,2\}$, and insert $(y+2c)/3$.
More explicitly, choose $j$ to make $y+2(v-jB_0)$ divisible by
three. If $g=y-v\ge5B_0$ and $z=(y+2c)/3$, then
\[
 z-v=\frac{g-2jB_0}{3}>0,\qquad
 z-v\le\frac g3,\qquad z<y.
\]
Here $c\le v<y$ proves the last inequality, and
$g-2jB_0\ge B_0>0$ proves the first. The point $c$ is one of
the last three progression points of $K_0$, so it is available.
Thus $(y,c,c)$ is a nontrivial certificate for an integer point
$z$ strictly between $v$ and $y$. Replacing $y$ by $z$ proves the
same positive factor-three decrease as on the left.
At the end fill $[v-B_0,y]$; the endpoint $v-B_0$ has progression
neighbours $v-2B_0,v$, and all other new interior points use
step-one neighbours. This repairs $v$. All these points belong
to $[0,R_0]$.

These certificates are simultaneous membership statements in the
final set. A point used as a witness need not have received its own
certificate earlier in the procedure. Thus using $u$ in the
certificate of $u+B_0$, while repairing $u$ by the filled interval,
introduces no logical dependency: both identities are verified
directly and all their entries belong to the final set.

It remains to count the two completion chains. If a chain performs
$k$ steps, its final positive integer gap is at most its initial
gap divided by $3^k$. Since that initial gap is less than the
terminal width $D$, we have $3^k<D$. As
\[
 D<60M^2=60\cdot3^{2t}<3^{2t+4},
\]
we get $k\le2t+3$. The chains and two filled intervals together
cost at most $4t+6+12B_0$ points. Including the core and initial
endpoints, their contribution is at most
\[
 5b+2+(4t+6)+12(b+1)=17b+4t+20.
\]
Adding the slot contribution gives \eqref{eq:sparse-size}.

The resulting budget can also be read stage by stage:
\[
 \begin{array}{l|c}
 \text{points charged}&\text{upper bound}\\ \hline
 \text{initial core and two endpoints}&5b+2\\
 \text{two completion chains}&4t+6\\
 \text{two filled intervals}&12(b+1)\\
 \text{transition slots}&
 2(t+1)\bigl(\lfloor\log_{3M^2}p\rfloor+1\bigr).
 \end{array}
\]
Each term has a separate purpose: the core supplies corrections,
the slots shorten the working interval, and the two completion
terms supply certificates at the extreme core points. In particular,
the completion cost depends on $t$ and $b$, rather than on the number
of earlier slots or on their starting widths. This separation is
what permits the final choice of scale.

All counts here are upper bounds on the size of a union of retained
points. A later insertion may coincide with an earlier point, in
which case it costs nothing extra. Such coincidences cannot
invalidate a certificate: the proof of nontriviality at each
creation step uses strict inequalities between its endpoint values,
and membership of all witness values persists in the final set.
This justifies summing the insertion budgets without proving that
the different stages produce disjoint sets.

Finally reduce the retained integer set modulo $p$. Reduction is
injective on $[0,R_0]\subset[0,p)$, so all distinctness and
nontriviality checks survive. Every point has one of the recorded
certificates: an original core certificate, an initial endpoint
certificate, a creation certificate, or a completion certificate.
The set is nonempty and hence is weakly ternary-balanced.
By Lemma~\ref{lem:ternary-transfer}, $C+C$ is NUS. The sum map
from unordered two-element multisets of $C$ onto $C+C$ gives
$|C+C|\le\binom{|C|+1}{2}$, proving the final assertion.
\end{proof}

\subsection{The second-order estimate}
\begin{proof}[Proof of Theorem~\ref{thm:main-upper}]
Write $X=\log p$. We choose the scale so that the core is small
compared with the desired second-order term, while the number of
contraction rounds still has the correct leading order. Put
\[
 s_X=\frac{2\log X-6\log\log X}{\log3},
 \qquad t=4\left\lfloor\frac{s_X}{4}\right\rfloor.
\]
The choice can be anticipated from the finite budget. Its transition
term has leading size
\[
 \frac{2(t+1)X}{(2t+1)\log3}
 =\frac X{\log3}+\frac{X}{(2t+1)\log3}.
\]
Increasing $t$ decreases the second summand, but increases the core
cost $b=3^{t/2}$. A scale close to $2\log X/\log3$ makes the
second summand about $X/(4\log X)$. Subtracting
$6\log\log X/\log3$ in $s_X$ makes the core small enough that
it does not affect that coefficient. The subtraction is negligible
relative to the leading size of $t$.
For all sufficiently large $p$, this is a positive multiple of
four with $t\ge4$. The rounding error is less than four, so
\[
 s_X-4<t\le s_X,
 \qquad t=\frac{2\log X}{\log3}+O(\log\log X).
\]
In particular $t\sim2\log X/\log3$. For the parameters in
Theorem~\ref{thm:sparse-core},
\[
 b=3^{t/2}\le\frac{X}{(\log X)^3},
 \qquad M_0=b^2\le\frac{X^2}{(\log X)^6}\le X^2
\]
eventually. Since $p=\ee^X>30X^2$ for all sufficiently large $X$,
the hypothesis $p>30M_0$ is satisfied. Thus the finite construction
applies to every sufficiently large prime, not just to a subsequence.

We now keep track of the error terms in its cardinality bound.
Since $\log(3M_0^2)=(2t+1)\log3$ and $\lfloor y\rfloor\le y$,
\eqref{eq:sparse-size} implies
\begin{align*}
 |C|&\le17b+4t+20+
       2(t+1)\left(\frac{X}{(2t+1)\log3}+1\right)\\
 &=\frac{2(t+1)X}{(2t+1)\log3}+17b+6t+22\\
 &=\frac X{\log3}
   +\frac{X}{(2t+1)\log3}+17b+6t+22.
\end{align*}
The identity $2(t+1)=(2t+1)+1$ separates the leading term from
the term that will produce the second coefficient. The remaining
terms are smaller than $X/\log X$: indeed
\[
 \frac{b\log X}{X}\le\frac1{(\log X)^2}\longrightarrow0,
 \qquad \frac{t\log X}{X}
       =O\left(\frac{(\log X)^2}{X}\right)\longrightarrow0.
\]
The constant term has the same vanishing property. On the other
hand, $t\sim2\log X/\log3$ gives
\[
 \frac{\log X}{(2t+1)\log3}\longrightarrow\frac14.
\]
Consequently there is a real function $\eta(X)\to1/4$ such that
the upper majorant
\[
 U(X)=\frac X{\log3}+\eta(X)\frac X{\log X}
\]
satisfies $|C|\le U(X)$ for all sufficiently large primes. This
is a bound on $|C|$, not an asserted asymptotic equality for the
size of the particular set constructed.

For example, an explicit choice of this error function is
\[
 \eta(X)=\frac{\log X}{(2t+1)\log3}
       +(17b+6t+22)\frac{\log X}{X},
\]
where $t$ and $b$ are the functions of $X$ defined above. The
first summand tends to $1/4$ and the second tends to zero by
the preceding estimates. Thus the use of $o(1)$ does not hide
a choice of scale depending on a subsequence of primes or on
the eventual value of $|C|$. Rounding $t$ down to a multiple
of four is already included in this formula and in its limit.

We can make the convergence in this step more explicit. Put
$L=\log X$. The definition and rounding of $t$ give
\[
 (2t+1)\log3=4L-12\log L+O(1).
\]
For large $L$ this denominator is bounded below by $2L$. Subtracting
$1/(4L)$ from its reciprocal therefore gives
\[
 \frac1{(2t+1)\log3}
 =\frac1{4L}+O\!\left(\frac{\log L}{L^2}\right).
\]
Indeed, the numerator in the difference of the two fractions is
$12\log L+O(1)$, and their denominator has order $L^2$.
Since $b\le X/L^3$ and $t=O(L)$, the displayed formula for
$\eta(X)$ now yields
\[
 \eta(X)=\frac14+
 O\!\left(\frac{\log L}{L}+\frac1{L^2}+\frac{L^2}{X}\right)
 =\frac14+O\!\left(\frac{\log L}{L}\right).
\]
In the last step $X=\ee^L$ was used. This computation identifies
the effects of scale choice, rounding, and the auxiliary core budget
separately; each has a vanishing contribution to the second
coefficient.

Set $a_0=\log3$. The function $u\mapsto u(u+1)/2$ is increasing
for $u\ge0$, so the transfer and finite counting bound yield
\begin{align*}
 m(p)&\le\frac{U(X)(U(X)+1)}2\\
 &=\frac{X^2}{2a_0^2}
   +\frac{\eta(X)}{a_0}\frac{X^2}{\log X}
   +\frac{\eta(X)^2}{2}\frac{X^2}{(\log X)^2}
   +\frac{X}{2a_0}+\frac{\eta(X)X}{2\log X}.
\end{align*}
Because $\eta(X)$ is bounded, the last three terms are all
$o(X^2/\log X)$. The second term has coefficient tending to
$1/(4\log3)$. Hence
\[
 m(p)\le\frac{X^2}{2(\log3)^2}
       +\left(\frac1{4\log3}+o(1)\right)
          \frac{X^2}{\log X}.
\]
In quantified form, for any prescribed positive error, the
coefficient of $X^2/\log X$ is eventually at most
$1/(4\log3)$ plus that error. Substituting $X=\log p$ and
$\log X=\log\log p$ proves the stated theorem.
\end{proof}

\section{Elementary bounds for all odd primes}\label{app:elementary}
For completeness, we record the finite estimate
\begin{equation}\label{eq:elementary}
 1+\lceil\log_2p\rceil\le m(p)
 \le\min\{p,2\lfloor\log_2p\rfloor^2+2\lfloor\log_2p\rfloor+1\}.
\end{equation}
The lower bound follows from balance, while the upper bound uses binary
prefixes and Lemma~\ref{lem:balanced-sum}. The balanced-set lower bound
is classical; see the introduction of Nedev~\cite{Nedev} and
Bedert~\cite[Corollary~2]{Bedert}. We include a determinant proof.

\begin{lemma}[Balance detects the prime]\label{lem:balanced-lower}
Every NUS set in odd characteristic is balanced. Every nonempty balanced $B\subseteq\F_p$ has
\[
 |B|\ge1+\lceil\log_2p\rceil.
\]
\end{lemma}
\begin{proof}
Let $A$ be NUS and $b\in A$. Repair the diagonal pair $\{b,b\}$
by a different pair $\{u,v\}$ with $2b=u+v$. If $u=b$,
cancellation gives $v=b$, contrary to the repair being different;
the same applies to $v=b$. If $u=v$, invertibility of $2$ in
odd characteristic gives $u=v=b$, also a contradiction. Thus
$u,v\in A\setminus\{b\}$ and $u\ne v$, proving balance.

For the numerical bound, choose an inclusion-minimal nonempty
balanced subset $B_0=\{b_1,\ldots,b_k\}$ of $B$. Necessarily
$k\ge3$, since a midpoint certificate needs two distinct other
points. At each $b_i$
choose one midpoint certificate with endpoint indices $u_i,v_i$,
so $2b_i=b_{u_i}+b_{v_i}$, where $u_i,v_i$ are distinct and
neither equals $i$. Form the integer matrix with rows
\[
 r_i=2e_i-e_{u_i}-e_{v_i}.
\]
We prove the required determinant properties directly, using the
incidence-minor lemma rather than a spanning-tree formula.

Suppose $f\in\Q^k$ is in the right kernel. Then
$2f_i=f_{u_i}+f_{v_i}$ for every $i$. If $f_i$ is maximal,
both terms on the right must have the same maximal value. The
maximum layer therefore contains the chosen midpoint witnesses
of each of its points and is itself a nonempty balanced subset
of $B_0$. Minimality makes it all of $B_0$. Hence $f$ is
constant. Conversely constants are in the kernel, so the matrix
has rational rank $k-1$.

Delete the column of $b_k$. Since all row sums are zero, deletion
is injective on the rational row space. Choose $k-1$ of the actual
rows whose restrictions form a nonsingular square matrix $D$.
Modulo $p$, the nonzero vector
$(b_1-b_k,\ldots,b_{k-1}-b_k)$ is in its right kernel: this is
the certificate identity with the anchor label subtracted.
Thus $p\mid\det D$, while $\det D\ne0$ as an integer, and
$p\le|\det D|$.

Each selected row has the decomposition
\[
 2e_i-e_{u_i}-e_{v_i}
 =(e_i-e_{u_i})+(e_i-e_{v_i}),
\]
also after deleting the anchor column. Multilinearity of the
determinant expands $\det D$ into $2^{k-1}$ terms, obtained
by choosing one of the two unit differences in each row.
Every term is an incidence minor and has absolute value at most
one by Lemma~\ref{lem:incidence-minors}. Therefore
\[
 p\le|\det D|\le2^{k-1}.
\]
Taking logarithms to base two gives $\log_2p\le k-1$. Since
$k-1$ is an integer, $\lceil\log_2p\rceil\le k-1$.
Finally $k\le|B|$, which proves the claimed bound.
\end{proof}

\begin{proposition}[All-prime binary construction]Let $p$ be an odd prime. For $L_2=\lfloor\log_2p\rfloor$ and $n_j=\lfloor p/2^j\rfloor$, put
\[
 B=\{0\}\cup\{\pm n_j:1\le j\le L_2\}.
\]
Then $B$ is balanced, $|B|=2L_2+1$, and
\[
 |B+B|\le\min\{p,2L_2^2+2L_2+1\}.
\]
\end{proposition}
\begin{proof}
Write $L=L_2$. Since $2^L\le p<2^{L+1}$, the last prefix is
$n_L=1$ and every $n_j$ with $1\le j\le L$ is positive.
Euclidean division by two gives, for $1\le j\le L$,
\[
 n_{j-1}=2n_j+\epsilon_{j-1},
 \qquad \epsilon_{j-1}\in\{0,1\},\qquad n_0=p.
\]
Here dividing $\lfloor p/2^{j-1}\rfloor$ by two has quotient
$\lfloor p/2^j\rfloor$, which justifies the recurrence with
the floor functions. Since $n_j\ge1$, the recurrence implies
$n_{j-1}>n_j$. The positive prefixes are thus distinct, and
$n_1=(p-1)/2$ because $p$ is odd.

They also remain distinct from all negative prefixes modulo $p$.
Indeed for $1\le j,k\le L$,
\[
 0<n_j+n_k\le p-1,
\]
so $n_j\not\equiv-n_k\pmod p$. None is zero modulo $p$.
It follows that $|B|=1+2L$.

We verify balance at every kind of point. For $n_1$, the recurrence
starts with $p=2n_1+1$, so
\[
 2n_1\equiv0+(-1)\pmod p.
\]
Both endpoints lie in $B$, since $n_L=1$, and they are distinct
from one another and from $n_1$. For $j\ge2$, use
\[
 2n_j=n_{j-1}+(-\epsilon_{j-1}).
\]
The first endpoint is an earlier positive prefix and exceeds
$n_j$. The second is either $0$ or $-1$, both in $B$.
Their integer representatives show that neither equals $n_j$
modulo $p$, and the endpoints are distinct modulo $p$ as well:
all positive prefixes lie between $1$ and $(p-1)/2$, whereas
the representatives of $0,-1$ are $0,p-1$. These are nontrivial
midpoint certificates. Negating them provides the certificates
for negative prefixes. Finally $2\cdot0=1+(-1)$ is nontrivial
in odd characteristic. Thus $B$ is balanced.

There are exactly
\[
 \binom{|B|+1}{2}=(2L+1)(L+1)
\]
unordered two-element multisets from $B$. At least $L+1$ distinct
such multisets have the same sum zero, namely $\{0,0\}$ and
$\{n_j,-n_j\}$ for $1\le j\le L$. Their distinctness follows
from the disjointness and strict ordering of the prefix lists.
Replacing these $L+1$ representations by their one sum saves
at least $L$ elements in the image of the sum map. Other collisions
can only reduce that image further. Hence
\[
 |B+B|\le(2L+1)(L+1)-L=2L^2+2L+1.
\]
The independent bound $|B+B|\le p$ follows from $B+B\subseteq\F_p$.
Together these give the stated minimum. Lemma~\ref{lem:balanced-sum}
also shows that $B+B$ is NUS, giving the upper half of
\eqref{eq:elementary}.
\end{proof}

\end{document}